\documentclass[english]{article}

\usepackage[T1]{fontenc}
\usepackage[utf8]{inputenc}
\usepackage{babel}

\usepackage{color}

\usepackage{graphicx}
\usepackage{amsmath}
\usepackage{amsthm}
\usepackage{amssymb}
\usepackage{amsfonts}
\usepackage{latexsym}

\definecolor{blue}{rgb}{0,0,0.8}
\definecolor{red}{rgb}{0.8,0,0}
\definecolor{darkgreen}{rgb}{0,0.6,0}
\definecolor{orange}{rgb}{0.98,0.5,0}

\newcommand{\dg}[1]{{\color{darkgreen}{#1}}}

\newcommand{\PP}{{\mathcal P}}
\newcommand{\ze}{\zeta}
\newcommand{\zs}{\zeta(s)}
\newcommand{\zsr}{\zeta(s+\rho)}

\newcommand{\ok}{\omega_k}
\newcommand{\oj}{\omega_j}

\newcommand{\DD}{{\mathcal D}}
\newcommand{\GG}{{\mathcal G}}

\newcommand{\si}{\sigma}

\newcommand{\G}{{\mathcal G}}

\newcommand{\Z}{{\mathcal Z}}

\newcommand{\RR}{{\mathbb R}}
\newcommand{\CC}{{\mathbb C}}

\newcommand{\NN}{{\mathbb N}}

\newcommand{\N}{{\mathcal N}}
\newcommand{\R}{{\mathcal R}}

\newcommand{\SSS}{{\mathcal S}}

\newcommand{\D}{{\Delta}}

\newcommand{\ve}{{\varepsilon}}

\newtheorem{theorem}{Theorem}
\newtheorem{corollary}{Corollary}
\newtheorem{lemma}{Lemma}

\newtheorem{definition}{Definition}

\theoremstyle{definition}

\newcommand{\de}{\delta}

\newcounter{othm}
\def\theothm{\Alph{othm}}
\newenvironment{othm}{
  \em
  \vskip 0.10in
  \refstepcounter{othm}
  \noindent{\bf Theorem\ \theothm .}
}{\vskip 0.10in}

\begin{document}

\title
{The  Carlson-type zero-density theorem for the Beurling $\zeta$ function\thanks{
Sz. Gy.~R\'{e}v\'{e}sz was supported in part by the Hungarian National Research,
Development and Innovation Office, project \#s K-132097, K-147153 and
K-146387.}}

\author{Szil\' ard Gy.~R\' ev\' esz}

\date{}

\maketitle

\begin{abstract}
In a previous paper we proved a Carlson-type  density theorem for zeroes in the critical strip for the  Beurling zeta functions satisfying Axiom A of Knopfmacher. There we needed to invoke two additonal conditions: the integrality of the norm (Condition B) and an ``average Ramanujan Condition'' for the arithmetical function counting the number of different Beurling integers of the same norm $m\in \NN$ (Condition G).

Here we implement a new approach of Pintz using the classic zero-detecting sums coupled with Halász' method, but otherwise arguing in an elementary way avoiding, e.g., large sieve-type inequalities or mean value estimates for Dirichlet polynomials. In this way we give a new proof of a Carlson-type  density estimate---with explicit constants---avoiding any use of the two additional conditions needed earlier.

Therefore it is seen that the validity of a Carlson-type density estimate does not depend on any
extra assumption---neither on the functional equation present for the Selberg class, nor on growth estimates of coefficients say of  ``average Ramanujan-type''---but is a general property presenting itself whenever the analytic continuation is guaranteed by Axiom A.
\end{abstract}



\medskip

\section{Introduction}

\subsection{Beurling's theory of generalized integers and primes}
Beurling's theory fits well to the study of several mathematical
structures. A vast field of applications of Beurling's theory is
nowadays called \emph{arithmetical semigroups}, which are
described in detail, e.g., by Knopfmacher in \cite{Knopf}.

Here $\G$ is a unitary, commutative semigroup with a countable set
of indecomposable generators called the \emph{primes} of $\G$ and
denoted usually as $p\in\PP$ (with $\PP\subset \G$ the set of all
primes within $\G$) which freely generate the whole of $\G$, i.e.,
any element $g\in \G$ can be (essentially, i.e., up to the order of
terms) uniquely written in the form $g=p_1^{k_1} \cdots p_m^{k_m}$; two (essentially) different  expressions of this form are necessarily different as elements of $\G$, while each element has
its (essentially) own unique prime decomposition.
Moreover, there is a \emph{norm} $|\cdot|: \G\to \RR_{+}$ so that
the following hold.

First, the image of $\G$, $|\G|\subset \RR_{+}$, is locally finite\footnote{Sometimes this property is mentioned as ``discreteness'', but what we mean is that any finite interval of $\RR_{+}$ can contain
the norm of only a finite number of elements of $\G$.}, hence the function
\begin{equation}\label{Ndef}
{\N}(x):=\# \{g\in \G~:~ |g| \leq x\}
\end{equation}
exists as a finite, nondecreasing, right-continuous, nonnegative integer-valued
function on $\RR_{+}$.

Second, the norm is multiplicative, i.e., $|g\cdot h| = |g| \cdot
|h|$; it follows that, for the unit element $e$ of $\G$, $|e|=1$ holds, and
that all other elements $g \in \G$ have norms strictly larger than 1.

Arithmetical functions can also be defined on $\G$. We will use in this work the M\"obius function $\mu$: its definition, which is analogous to the classical case see pages 36--37 in \cite{Knopf}. The generalized von Mangoldt function $\Lambda_{\GG}(g)$ will appear below in \eqref{vonMangoldtLambda}.

In this work we assume \emph{Axiom A} (in its normalized form to $\delta=1$) of
Knopfmacher, see page 75 (and for the normalization pages 78--79) of his fundamental book \cite{Knopf}.

\begin{definition} We say that ${\N}$ (or, loosely speaking, $\ze$)
satisfies \emph{Axiom A} --- more precisely, Axiom
$(\kappa, A, \theta)$ with suitable constants $\kappa, A>0$ and
$0<\theta<1$ --- if we have\footnote{The usual formulation uses the more natural version $\R(x):= \N(x)-\kappa x$. However, our version is more convenient with respect to the initial values at 1, as we here have $\R(1-0)=0$. All respective integrals of the form $\int_X$ will be understood as integrals from $X-0$, and thus we can avoid considering endpoint values in the partial integration formulae involving integration starting form 1. Alternatively, we could have taken also $\N(x):=\# \{g\in \GG, |g|<x\}$ left-continuous, and
$$
\R(x):=\N(x)-\begin{cases}\kappa x\dg{,} \qquad &\text{if}~ x>1\dg{,} \\ 0\dg{,} &\text{if} ~ x\le 1. \end{cases}
$$
Also with this convention we would have $\R(1-0)=0$ for the remainder, but this seemed to be less convenient than our choice.} for the remainder term
$$
\R(x):= \N(x)-\kappa (x-1)
$$
the estimate
\begin{align}\label{Athetacondi}
\left| \R(x) \right|  \leq A x^{\theta} \quad ( x \geq 1 ).
\end{align}
\end{definition}

It is clear that under Axiom A the Beurling zeta function
\begin{equation}\label{zetadef}
\ze(s):=\ze_{\G}(s):=\int_1^{\infty} x^{-s} d\N(x) = \sum_{g\in\G} \frac{1}{|g|^s}
\end{equation}
admits a meromorphic, essentially analytic continuation
$\kappa\frac{1}{s-1}+\int_1^{\infty} x^{-s} d\R(x)$ up to $\Re s >\theta$
with only one simple pole at 1.

\subsection{Analytic theory of the distribution of Beurling generalized primes}

The Beurling $\zeta$ function \eqref{zetadef} makes it possible to express the values of
the generalized von Mangoldt function
\begin{equation}\label{vonMangoldtLambda}
\Lambda (g):=\Lambda_{\G}(g):=\begin{cases} \log|p| \quad & \textrm{if}\quad g=p^k,
~ k\in\NN ~~\textrm{with some prime}~~ p\in\PP\\
0 \quad &\textrm{if}\quad g\in\G ~~\textrm{is not a prime power in} ~~\G
\end{cases},
\end{equation}
as coefficients of the logarithmic derivative of the zeta function:
\begin{equation}\label{zetalogder}
-\frac{\zeta'}{\zeta}(s) = \sum_{g\in \G} \frac{\Lambda(g)}{|g|^s}.
\end{equation}
Beurling's theory of generalized primes is mainly concerned with the analysis of the summatory function
\begin{equation}\label{psidef}
\psi(x):=\psi_{\G}(x):=\sum_{g\in \G,~|g|\leq x} \Lambda (g).
\end{equation}
The generalized PNT (Prime Number Theorem) is the asymptotic equality $\psi(x)\thicksim x$. The remainder term in this equivalence is denoted, as usual,
\begin{equation}\label{Deltadef}
\Delta(x):=\Delta_{\G}(x):=\psi(x)-x.
\end{equation}
In the classical case of prime number distribution, as well as regarding some extensions to primes in arithmetical progressions and
the distribution of prime ideals in algebraic number fields, the connection between
the location and distribution of zeta-zeroes and the oscillatory behavior of the remainder term $\D(x)$ in the prime number formula $\psi(x)\thicksim x$ is well understood \cite{Kac1, K-97, K-98, Knapowski, Pintz1, Pintz2, Pintz9, PintzProcStekl, Rev1, Rev2, RevPh, RevAA, Stas1, Stas2, Stas-Wiertelak-1, Stas-Wiertelak-2, Turan1, Turan2}.
On the other hand, in the generality of Beurling primes and the zeta function, investigations so far were focused on mainly four directions. First, better and better
minimal conditions were sought in order to have a Chebyshev-type formula $x\ll \psi(x) \ll x$, see, e.g., \cite{Vindas12, Vindas13, DZ-13-2, DZ-13-3}. Understandably, as in the classical case, this relation requires only an analysis of the $\zeta$ function of Beurling in and on the boundary of the convergence half-plane. Second, conditions for the PNT to hold were sought, see, e.g., \cite{Beur, K-98, DebruyneVindas-PNT, {DebruyneVindas-RT}, DSV, DZ-17, Zhang15-IJM, Zhang15-MM}. Again, this relies on the boundary behavior of $\zeta$ on the one-line $\si=1$. Third, rough (as compared to our knowledge in the natural prime number case) estimates and equivalences were worked out in the analysis of the connection between the $\zeta$-zero distribution and the error term behavior for $\psi(x)$, see, e.g. \cite{H-5}, \cite{H-20}.
Fourth, examples were constructed for arithmetical semigroups with very ``regular'' (such as satisfying the Riemann Hypothesis RH and the error estimate  $\psi(x)=x+O(x^{1/2+\varepsilon})$) and very ``irregular'' (such as having no better zero-free regions than \eqref{classicalzerofree} below and no better asymptotic error estimates than \eqref{classicalerrorterm}) behavior and zero or prime distribution, see, e.g., \cite{H-15}, \cite{BrouckeDebruyneVindas}, \cite{DMV}, \cite{H-5}, \cite{Zhang7}. Here we must point out that the above citations are just examples and are far from being a complete description of the otherwise formidable literature
\footnote{E.g., a natural, but somewhat different direction, going back to Beurling himself, is the study of analogous questions in case the assumption of Axiom A is weakened to e.g. an asymptotic condition on $\N(x)$ with a product of $x$ and a sum of powers of $\log x$, or sum of powers of $\log x$ perturbed by almost periodic polynomials in $\log x$, or $\N(x)-cx$ periodic, see \cite{Beur}, \cite{Zhang93}, \cite{H-12}, \cite{RevB}.}.
For a thorough analysis of these directions as well as for much more information the reader may consult the monograph \cite{DZ-16}.

The main focus of our study presented in the recent papers \cite{Rev-One} and \cite{Rev-Many} was to establish as precisely as possible connections between the distribution of the zeros of the Beurling zeta function $\zeta$ on the one hand and the order of magnitude estimates or oscillatory properties of $\Delta(x)$ on the other hand.

Apart from generality and applicability to, e.g., distribution of prime ideals in number fields, the interest in the Beurling theory
 was greatly boosted by a construction of Diamond, Montgomery and Vorhauer \cite{DMV}. They basically showed\footnote{Let us call attention to the very nice further sharpening of this breakthrough result, which appeared very recently \cite{BrouckeVindas}.} that under Axiom A RH may still fail; moreover, nothing better than the most classical zero-free region and error term \cite{V} of
\begin{equation}\label{classicalzerofree}
\zeta(s) \ne 0 \qquad \text{whenever}~~~ s=\sigma+it, ~~ \sigma >
1-\frac{c}{\log t},
\end{equation}
and
\begin{equation}\label{classicalerrorterm}
\psi(x)=x +O(x\exp(-C\sqrt{\log x}))
\end{equation}
follows from \eqref{Athetacondi} at least if $\theta>1/2$.

\subsection{Carlson-type density estimates for the Beurling $\zeta$ function}

In \cite{Rev-D} we proved a Carlson-type density result for the zeros of the Beurling zeta function.
We needed this for our studies of the Littlewood- and Ingham-type questions studied in the Beurling context in our recent works \cite{Rev-One, Rev-Many}, for prior to \cite{Rev-D} no density estimates were known for the Beurling zeta function.

A predecessor of such results---the only one in the Beurling context which touched upon the topic of density-type estimates---was worked out by Kahane \cite{K-99}, who proved that,  under a suitable (strong) condition on the prime counting function, the number of Beurling zeta zeroes lying on some vertical line $\Re s=\si=a>\max(1/2,\theta)$ has finite upper density. That is already a nontrivial fact\footnote{This particular result enabled Kahane to draw deep number theoretical consequences regarding the oscillation (sign changes) of the error term in the prime number formula. Obviously, obtaining a much sharper result -- estimating the total number of zeroes in a full rectangle, not only on one individual vertical line, and with a quantity essentially below the order of $T$ when $a$ is getting close to $1$ -- provides an even stronger foothold for deriving number theoretical consequences.} because the total number $N(T)$ of zeroes with imaginary part not exceeding $T$ may grow in the order $T\log T$.

For deriving the density theorem below in \cite{Rev-D} we needed two additional assumptions. One was that the norm would actually map to the natural integers. Following Knopfmacher, this was called \emph{Condition B}. So we said that Condition B is satisfied if $|\cdot|:\G\to\NN$, that is, the norm $|g|$ of any element $g\in\G$ is a natural number. That was necessary mainly for using some large sieve type estimates from the classic book of Montgomery \cite{Mont}. Without this condition, the terms of the arising generalized Dirichlet polynomials occurring in our proof could not be controlled well, and such strong tools could not be used.

As is natural, we will write $\nu\in|\G|$ if there exists $g\in\G$ with $|g|=\nu$. Under Condition B we can introduce the arithmetical function $G(\nu):=\sum_{g\in\G,~|g|=\nu} 1$, which is then a super-multiplicative\footnote{That is, $G(\mu \nu) \ge G(\mu) G(\nu)$ whenever $(\mu,\nu)=1$.} arithmetical function on $\NN$. The next condition, called \emph{Condition G} and also taken from \cite{Knopf}, was a so-called ``average Ramanujan condition'',
meaning that the arithmetical function $G(\nu)$ is $O(\nu^\ve)$, at least on the (say $p$-th power) average.

Denote the number of zeroes of the Beurling zeta function in $[b,1]\times [-iT,iT]$ as
\begin{equation}\label{NTest}
N(b,T):=\#\{ \rho=\beta+i\gamma~:~ \ze(\rho)=0, \,\beta\geq b,
|\gamma|\leq T \}.
\end{equation}

The main result of the paper \cite{Rev-D} was the following.
\begin{othm}\label{th:density} Assume that $\G$ satisfies besides Axiom A also Conditions B and G, too. Then for any $\varepsilon>0$ there exists a constant $C=C(\varepsilon,\G)$ such that for all sufficiently large $T$ and\footnote{Note that, by the (standard) Lemma \ref{l:Littlewood} below, $N(\alpha,T)=O(T^{1+\ve})$ for $\alpha>\theta$ always. Thus, the statement is nontrivial only if $\alpha$ is close to $1$, more precisely when $\alpha> \frac{5-\theta}{6-2\theta}$.}
$\alpha>(1+\theta)/2$ we have
\begin{equation}\label{density}
N(\alpha,T)\leq C T^{\frac{6-2\theta}{1-\theta}(1-\alpha)+\ve}.
\end{equation}
\end{othm}

Theorem \ref{th:density} was somewhat surprising, because we lack a functional equation essential in the treatment of the Selberg class, where zero density estimates are known to hold \cite{KP}. However, as one referee pointed out to us, the functional equation is mainly used in the Selberg class to estimate $\zeta$, and so we could possibly succeed only because similar estimates can be derived directly from our extra conditions.

Two main questions can be naturally posed after this result. First, if solely under Axiom A such a density result can be obtained. Here we give an affirmative answer to this question\footnote{Thus, we may surprise even the above mentioned referee.}.

\begin{theorem}\label{th:NewDensity} Let $\G$ be a Beurling system subject to Axiom $A$. Then for any $\si>(1+\theta)/2$ the number of zeroes of the corresponding Beurling zeta function $\zs$ admits a Carlson-type density estimate
\begin{equation}\label{densityresult}
N(\si,T) \le  1000 \frac{(A+\kappa)^4}{(1-\theta)^3 (1-\si)^4} T^{\frac{12}{1-\theta}(1-\si)} \log^5 T
\end{equation}
for all $T\ge T_0$, where also $T_0$ depends explicitly on the parameters $A, \kappa, \theta$ of Axiom A and on the value of $\eta:=1-\si$. In particular, for $\si>\frac{11+\theta}{12}$ we have $N(\si,T)=o(T)$.
\end{theorem}

This means that even without any extra assumptions, order or regularity, just by the mere Axiom A, (which is the natural assumption to guarantee the analytic continuation of $\zs$ to some larger half-plane $\Re s >\theta$ with $\theta<1$), a Carlson-type density theorem holds true, always.
So, in a philosophical sense, density results hold not because of some extra assumptions, but because of the basic analytical nature of the Beurling $\zeta$ function; in particular, the role of a functional equation can fully be suspended for such a result to hold.

Second, in case of the classical Riemann zeta function, advanced Vinogradov type estimates and corresponding zero-free regions (and other advances) were all exploited in getting even stronger density results with $o(1-\si)$ exponents at least in the vicinity of $\si=1$, the very first such result being achieved by Halász and Turán \cite{Hal-Tur-I, Hal-Tur-II}. For further advances on this issue of key number theory importance see, e.g., \cite{Bombieri, HB, PintzBestDens} and in particular \cite{GGL} and \cite{Pintz-AMH-Dens}, whose direct and simplified method we follow here to a great extent. As the fundamental work \cite{DMV} demonstrated, however, Vinogradov-type strong estimates cannot be expected in the generality of Beurling systems, as in particular it can well happen that only the classical de la Vallée-Poussin--Landau-type zero free region \eqref{classicalzerofree} exists and only the classical de la Vallée-Poussin error estimate \eqref{classicalerrorterm} holds true. So the natural question regarding density results is: are Carlson-style $O(T^{C(1-\si)+\ve})$ estimates optimal (at least ``in their nature'', leaving the still important question of the value of $C$ still subject to further optimization), or there can be expected some Halász--Turán-type $T^{o(1-\si)+\ve}$ improved density estimate, too?

This second question has already been answered, too. Indeed, F. Broucke and G. Debruyne  \cite{BrouckeDebruyne} constructed a Beurling system
(by a nontrivial adaptation of the Diamond--Montgomery--Vorhauer construction) subject to Axiom A, but having $\Omega(T^{c(1-\si)})$ zeroes in some rectangle $[\si,1]\times [-iT,iT]$ with $\si$ arbitrarily close to $1$. That is, Carlson-type density estimates might be seen in general as the best possible in the generality of Beurling systems satisfying Axiom A.

\subsection{Consequences of density theorems regarding the error term of the PNT of Beurling}

In the recent papers \cite{Rev-One, Rev-Many} we investigated the questions of Littlewood and Ingham together with their converse about the connection of the location of $\zeta$-zeros and order resp.~oscillation estimates for the error term $\D(x)$ in the PNT of Beurling.
As we said before, to obtain our results originally (e.g. in the first ArXiv version) we needed to use Theorem \ref{th:density} at several occurrences. Therefore, several of our results were restricted to Beurling systems and zeta functions satisfying also Conditions B and G, moreover,
the implied constants also depended on the somewhat implicit ones of these conditions and the density result of Theorem \ref{th:density}.

Extending the generality of the density theorem serves also to extend several of our number theoretical results regarding the connection between the behavior of $\Delta(x)$ on the one hand and zero distribution of $\zeta(s)$ on the other hand.
We briefly discuss these improved versions of our results from \cite{Rev-One, Rev-Many} in the last section, also pointing out the explicit dependence of constants on the main parameters $A, \kappa, \theta$ of the system. These results have been incorporated into \cite{Rev-Many} during its revision, therefore here they are appearing not for the first time, but it may be good to see how the improvement of the density theorem improved some related number theoretical results as well.

\section{Lemmata on the Beurling $\zeta$ function}\label{sec:basics}

The following basic lemmas are just slightly more explicit forms of well-known basic estimates, like, e.g., 4.2.6.
Proposition, 4.2.8. Proposition and 4.2.10. Corollary of
\cite{Knopf}. In \cite{Rev-MP} we elaborated on the proofs of some of them only for explicit
handling of the arising constants in these estimates. Those which did not appear in \cite{Rev-MP}, we briefly prove here without considering them original.

\subsection{Basic estimates of the Beurling $\zeta$ function and its partial sums}

\begin{lemma}\label{l:oneperzeta} For any $s=\si+it$, $\Re s=\si >1$ we have
\begin{equation}\label{zsintheright}
|\zs| \leq \ze(\si) \le \frac{(A+\kappa)\sigma}{\sigma-1} \qquad (\sigma >1),
\end{equation}
and also
\begin{equation}\label{reciprok}
|\zs| \geq \frac{1}{\ze(\sigma)} >
\frac{\sigma-1}{(A+\kappa)\sigma} \qquad (\sigma >1).
\end{equation}
\end{lemma}

\begin{lemma}\label{l:zetaxs} Denote the ``partial sums'' (partial
Laplace transforms) of $\N|_{[1,X]}$ as $\ze_X$ for arbitrary $X\geq 1$:
\begin{equation}\label{zexdef}
\ze_X(s):=\int_1^X x^{-s} d\N(x).
\end{equation}
Then $\ze_X(s)$ is an entire function and for $\sigma:=\Re s
>\theta$ it admits
\begin{equation}\label{zxrewritten}
\ze_X(s)=
\begin{cases} \ze(s)-\frac{\kappa X^{1-s}}{s-1}-\int_X^\infty x^{-s}d\R(x)
& \textrm{for all} ~~~s\ne 1,  \\
\frac{\kappa }{s-1}-\frac{\kappa X^{1-s}}{s-1}+\int_1^X
x^{-s}d\R(x)
& \textrm{for all} ~~~s\ne 1, \\
\kappa \log X + \int_1^X \frac{d\R(x)}{x} & \textrm{for}~~~ s=1,
\end{cases}
\end{equation}
together with the estimate
\begin{equation}\label{zxesti}
\left|\ze_X(s) \right| \leq \ze_X(\sigma) \leq
\begin{cases} \min \left( \frac{\kappa X^{1-\sigma}}{1-\sigma}
+ \frac{A}{\sigma-\theta},~ \kappa X^{1-\sigma}\log X +
\frac{A}{\sigma-\theta}\right) &\textrm{if} \quad
\theta<\sigma < 1,
\\ \kappa \log X + \frac{A}{1-\theta}& \textrm{if} \qquad \sigma =1,
\\ \min\left( \frac{\sigma (A+\kappa)}{\sigma-1},~{\kappa}\log X + \frac{\sigma A}{\sigma-\theta}
\right) &\textrm{if}\quad
\sigma>1.
\end{cases}
\end{equation}
Moreover, the above remainder terms can be bounded as follows:
\begin{equation}\label{zxrlarge}
\left| \int_X^\infty x^{-s}d\R(x) \right| \leq A
\frac{|s|+\sigma-\theta}{\sigma-\theta} X^{\theta-\sigma},
\end{equation}
and
\begin{equation}\label{zxrlow}
\left| \int_1^X x^{-s}d\R(x) \right| \leq A \left(
|s|\frac{1-X^{\theta-\sigma}}{\sigma-\theta} +
X^{\theta-\sigma}\right) \leq A \min \left(
\frac{|s|}{\sigma-\theta},~|s| \log X + X^{\theta-\sigma} \right).
\end{equation}
\end{lemma}

\begin{lemma}\label{l:zetasigma} For any $X\ge 1$ and $s=\si+it$ with $\Re s=\si <1 $, it holds
\begin{equation}\label{zetaXlogfree}
|\zeta_X(s)| \le \zeta_X(\si) \le \frac{A+\kappa}{1-\si} X^{1-\si}
\end{equation}
and also
\begin{equation}\label{zetaXestimate}
|\zeta_X(s)|   \le \zeta_X(\si) \le \frac{A+\kappa}{\si-\theta} X^{1-\si}\log X,
\end{equation}
\end{lemma}
\begin{proof} We have
\begin{align*}
|\zeta_X(s)| \le \zeta_X(\si)&=\int_1^X \frac{d \N(t)}{t^{\si}} = [\N(t)t^{-\si} ]_1^X +\si \int_1^X \frac{\N(t)}{t^{\si+1}} dt.
\end{align*}
The integrated first term is $\N(X) X^{-\si}\le (A+\kappa) X^{1-\si}$. For the second term we find
$$
\si \int_1^X \frac{\N(t)}{t^{\si+1}} dt \le \si \int_1^X \frac{A+\kappa}{t^{\si}} dt = (A+\kappa)\frac{\si}{1-\si} \left(X^{1-\si}-1\right)< (A+\kappa)\frac{\si}{1-\si} X^{1-\si}.
$$
Adding the two estimates, we get \eqref{zetaXlogfree}.

The second part of the first line of \eqref{zxesti} directly implies \eqref{zetaXestimate}.
\end{proof}

\subsection{Behavior of the Beurling zeta function in the critical strip}\label{sec:orderofgrowth}

\begin{lemma}\label{l:zkiss} We have
\begin{equation}\label{zsgeneral}
\left|\zs-\frac{\kappa}{s-1}\right|\leq \frac{A|s|}{\sigma-\theta}
\qquad \qquad\qquad (\theta <\sigma ,~ t\in\RR,~~ s\ne 1).
\end{equation}
In particular, for large enough values of $t$ it holds
\begin{equation}\label{zsgenlarget}
\left|\zs \right|\leq \sqrt{2} \frac{(A+\kappa)|t|}{\sigma-\theta}
\qquad \qquad\qquad (\theta <\sigma \leq |t|),
\end{equation}
while for small values of $t$ we have
\begin{equation}\label{zssmin1}
|\zs(s-1)-\kappa|\leq \frac{A|s||s-1|}{\sigma-\theta} \leq
\frac{100 A}{\sigma-\theta}\qquad (\theta <\sigma \leq 4,~|t|\leq
9).
\end{equation}
As a consequence, we also have
\begin{equation}\label{polenozero}
\zs\ne 0 \qquad \textrm{for}\qquad  |s-1| \leq
\frac{\kappa(1-\theta)}{A+\kappa}.
\end{equation}
\end{lemma}

\begin{lemma}\label{l:zzzz} For arbitrary $X \ge 2$ and $s=\si+it$, $\theta<\Re s=\si<1$ and $t\ge 2$ the following estimates hold true:
\begin{align}
|\zeta(s) -\zeta_X(s)| & \le \frac{1}{\si-\theta} \left( \kappa \frac{X^{1-\si}}{t}+ 2A \frac{t}{X^{\si-\theta}} \right) \le  \frac{2A + \kappa}{\si-\theta} \left( \frac{X^{1-\si}}{t}+ \frac{t}{X^{\si-\theta}} \right), \label{zetaminuszetaX}
\\
|\zeta(s)| &  \le \frac{(2A+\kappa)(1-\theta)}{(1-\si)(\si-\theta)}  t^{\frac{1-\si}{1-\theta}}, \label{zetaestimate}
\\ |\zeta(s)| &  \le \frac{4A+3\kappa}{(\si-\theta)(1-\theta)}t^{\frac{1-\si}{1-\theta}} \log t. \label{zetawithlog}
\end{align}
In particular, by the triangle inequality, it also holds
\begin{equation}\label{zetaXsecond}
|\zeta_X(s)|   \le  \min\left(\frac{(2A+\kappa)(1-\theta)}{(1-\si)(\si-\theta)} t^{\frac{1-\si}{1-\theta}}, \frac{4A+3\kappa}{(\si-\theta)(1-\theta)}t^{\frac{1-\si}{1-\theta}} \log t \right)+ \frac{2A + \kappa}{\si-\theta} \left(\frac{X^{1-\si}}{t}+ \frac{t}{X^{\si-\theta}} \right).
\end{equation}
\end{lemma}
\begin{proof}

For estimating $\zs-\zeta_X(s)$ we combine the first line of \eqref{zxrewritten} and \eqref{zxrlarge} to infer
\begin{align*}
|\zs-\zeta_X(s)|&\le \frac{\kappa}{|s-1|}X^{1-\si} + A \frac{|s|+\si-\theta}{\si-\theta} X^{\theta-\si}
\le \frac{\kappa}{t} X^{1-\si}+ A\frac{t+2\si-\theta}{\si-\theta} X^{\theta-\si}.
\end{align*}
Taking into account $2\si-\theta <2 \le t$, this also furnishes \eqref{zetaminuszetaX}.

Further, substituting into the first form of \eqref{zetaminuszetaX} the particular choice $X:=t^{\frac{1}{1-\theta}} (\ge t \ge 2)$ leads to
$$
|\zs-\zeta_X(s)|\le \frac{2A+\kappa}{\si-\theta} ~t^{\frac{1-\si}{1-\theta}}.
$$
From here a trivial triangle inequality and \eqref{zetaXlogfree} of Lemma \ref{l:zetasigma} yield \eqref{zetaestimate}:
\begin{align*}
|\zs| & \le |\zs-\zeta_X(s)|+ |\zeta_X(s)| \le \frac{2A+\kappa}{\si-\theta} ~t^{\frac{1-\si}{1-\theta}}+ \frac{A+\kappa}{1-\si} t^{\frac{1-\si}{1-\theta}} < (2A+\kappa) \left( \frac{1}{\si-\theta} +\frac1{1-\si}\right) t^{\frac{1-\si}{1-\theta}}.
\end{align*}
If we apply here \eqref{zetaXestimate} instead of \eqref{zetaXlogfree} then we get
$$
|\zs| \le \frac{2A+\kappa}{\si-\theta} ~t^{\frac{1-\si}{1-\theta}}+ \frac{A+\kappa}{\si-\theta} t^{\frac{1-\si}{1-\theta}} \frac{1}{1-\theta} \log t < \frac{4A+3\kappa}{(\si-\theta)(1-\theta)}t^{\frac{1-\si}{1-\theta}} \log t,
$$
whence \eqref{zetawithlog}, too.
\end{proof}

Let us introduce the notation
\begin{align}\label{Mdef}
M(\si,T)&:=\max\{ |\zs|~:~ s=\si+it, \Re s=\si, 2\le |t|\le T\}.
\end{align}
The combination of  \eqref{zsgeneral}, \eqref{zetaestimate} and \eqref{zetawithlog} leads to the following.
\begin{corollary}\label{c:Mesti} For arbitrary $s=\si+it$ with $\Re s=\si \in (\theta,1)$ and for arbitrary $T\ge 0$ we have
\begin{equation}\label{Mestimate}
M(\si,T) \le \min\left(\frac{(2A+\kappa)(1-\theta)}{(1-\si)(\si-\theta)}, \frac{4A+3\kappa}{(\si-\theta)(1-\theta)} \log T \right) \max(1,T^{\frac{1-\si}{1-\theta}}) .
\end{equation}
\end{corollary}

\subsection{Estimates for the number of zeros of  $\zeta$}\label{sec:zeros}

Denote the set of all $\zeta$-zeroes in the rectangle $[b,1]\times i[-T,T]$ as $\Z(b;T)$
and the set of $\zeta$-zeroes in $[b,1]\times i([-T,-R]\cup[R,T])$ as
$\Z(b;R,T)$, while their cardinality is denoted by $N(b;T)$ and $N(b;R,T)$, respectively.
Also, we will write $\Z_{+}(b;R,T)$ for the part of $\Z(b,;R,T)$ lying in the upper half-plane.
\begin{lemma}\label{l:Littlewood} Let $\theta<b<1$ and consider
any height $T\geq 5$. Then the number of zeta-zeros $N(b,T)$ satisfy
\begin{equation}\label{zerosinth-corr}
N(b,T)\le \frac{1}{b-\theta}
\left\{\frac{1}{2} T \log T + \left(2 \log(A+\kappa) + \log\frac{1}{b-\theta} + 3 \right)T\right\}.
\end{equation}
\end{lemma}

\begin{lemma}\label{c:zerosinrange}
Let $\theta<b<1$ and consider any heights $T>R\geq 5$.
Then $N(b;R,T)$ satisfies\footnote{Here and below in \eqref{zerosbetweenone} a formulation with slightly different numerical constants is presented correcting the original calculation of \cite{Rev-MP}. About the error made in \cite{Rev-MP} and the description of the corrections see \cite{Rev-D}.}
\begin{equation}\label{zerosbetween}
N(b;R,T) \leq\frac{1}{b-\theta} \left\{ \frac{4}{3\pi} (T-R) \left(\log\left(\frac{11.4 (A+\kappa)^2}{b-\theta}{T}\right)\right)  + \frac{16}{3}  \log\left(\frac{60 (A+\kappa)^2}{b-\theta}{T}\right)\right\}.
\end{equation}

In particular, for the zeroes between $T-1$ and $T+1$ we have for
$T\geq 6$
\begin{align}\label{zerosbetweenone}
N(b;T-1,T+1) \leq \frac{1}{(b-\theta)} \left\{6.2 \log T +
6.2 \log\left( \frac{(A+\kappa)^2}{b-\theta}\right) + 24 \right\}.
\end{align}
\end{lemma}

\section{Proof of the density estimate of Theorem \ref{th:NewDensity} }\label{sec:ProofofNewDensity}

The constants $A$ and $\kappa$ frequently occur in our calculations and sometimes we take logarithms and reciprocals of their sum $A+\kappa$. To overcome technical distinctions and difficulties when, e.g., the logarithm is negative, let us note that once a constant $A$ is admissible, then the possibly enlarged value $A^*:=\max(A,1-\kappa)$ is admissible, too, so that in what follows we will automatically consider that $A+\kappa\ge 1$.

At the outset we fix some parameter $\xi$ with $\theta<\xi<\sigma(<1)$, and also write $\eta:=1-\sigma$ and $\delta:=\si-\xi$, so that, e.g., $1-\xi=\de+\eta$. We will need later on the restriction $\xi>\frac{1+\theta}{2}$ anyway, so we assume that once and for all.
Further, note that the statement \eqref{densityresult} of Theorem \ref{th:NewDensity} directly follows from Lemma \ref{l:Littlewood} if $\eta > \frac{1}{12}(1-\theta)$, so that we can assume in the following that $\eta\le (1-\theta)/12\le 1/12$. In the course of the
proof we will finally specify $\de$ as $\de=1.5 \eta$, so that we will also have $\de\le(1-\theta)/8\le 1/8$. The numerical estimates $\eta\le 1/12$ and $\de\le 1/8$ will thus be capitalized on without further mention\footnote{Also we will use without notice that $\log u \le \frac{1}{e} u$ and $\log u \le \frac{2}{e} \sqrt{u}$, always.}. Let us also note that with these assumptions $\xi=1-2.5\eta>\frac{1+\theta}{2}$ is guaranteed, too.

Further, we take three large parameters $X>Y>e^{10}$ and $T>e^{10}$, and denote $\lambda:=\log Y$, $L:=\log T$. In fact, in two steps we will restrict $X$ first to be at least $e^2Y$, and then to be exactly this value, so that we can as well consider $X:=e^2Y$ right away. (We can also foretell that we will choose $X:=T^{\frac{3.5}{1-\theta}}$ at the end.)

\subsection{The zero detecting method}
The starting point of the proof is the following quantity defined by a complex integral:
$$
I:=\frac{1}{2\pi i} \int_{(\Re s=3)} \frac{1}{s} \exp\left(s^2/L+\lambda s\right)~ds=1+\frac{1}{2\pi i} \int_{(-L)} \frac{1}{s} \exp\left(s^2/L+\lambda s\right)~ds,
$$
where the last formula can be obtained by the Residue Theorem upon shifting the line of integration to the left until $\Re s=-L$. From this second expression it follows\footnote{The exact value of $I$ is of no importance for us here, but it is clear that its limit for either $L=\log T\to \infty$ or for $\lambda=\log Y\to \infty$ is 1.}
\begin{equation}\label{Irhofirstform}
\left|I-1\right| \le \frac{1}{\pi} \int_0^\infty \frac{1}{L} e^{L-t^2/L-\lambda L} dt = \frac{1}{\pi L} e^{-(\lambda-1)L} \frac{\sqrt{\pi L}}{2} <  T^{-(\lambda-1)} <0.1.
\end{equation}

The very simple base idea of the proof, with which we copy 
Pintz \cite{Pintz-AMH-Dens}, is to write here $\displaystyle 1=\frac{1}{\zsr} \zsr$, and thus involve a $\zs$-dependent (zero-detecting) expression into the simple mean $I$ defined above. Namely, for any complex number $\rho=\beta+i\gamma$ in the critical strip, i.e., with $\Re \rho=\beta \in (\theta,1)$, we now write in this trivial identity and reformulate it as follows:
\begin{align*}
I:=\frac{1}{2\pi i} \int_{(3)} \frac{1}{s} \exp\left(s^2/L+\lambda s\right)~ds &= \frac{1}{2\pi i} \int_{(3)} \frac{1}{\zsr} \frac{\zsr}{s} \exp\left(s^2/L+\lambda s\right)~ds
\\
&= \frac{1}{2\pi i} \int_{(3)} \sum_{g \in \G} \frac{\mu(g)}{|g|^{s+\rho}} \frac{\zsr}{s} \exp\left(s^2/L+\lambda s\right)~ds
\\
&= \sum_{g \in \G} \frac{\mu(g)}{|g|^{\rho}} \frac{1}{2\pi i} \int_{(3)} \frac{\zsr}{s} \exp\left(s^2/L+(\lambda-\log|g|) s\right)~ds.
\end{align*}
Later we will need three more restrictions on the choice of $\rho$: it will be restricted to the strip $\Re \rho=\beta \in (\si,1)$, its imaginary part will be chosen to satisfy $\Im \rho=\gamma \in [10L,T]$, and finally we will also assume that it is a zero of the Beurling zeta function $\zs$. That is, altogether we will have $\rho \in \Z_{+}(\si; 10L,T)$. However, for the moment we do not need these restrictions yet.

Defining for arbitrary $h\in \RR$ the weight function
$$
w(\rho,h):= \frac{1}{2\pi i} \int_{(3)} \frac{\zsr}{s} \exp\left(s^2/L+h s\right)~ds,
$$
the last expression of $I$ takes the form
\begin{equation}\label{Iwithw}
I=\sum_{g \in \G} \frac{\mu(g)}{|g|^{\rho}} w(\rho, \lambda-\log|g|).
\end{equation}

\subsection{Evaluation of terms with $w(\rho, h)$ where  $h$ is small (negative)}

Our next aim is to evaluate the weighted quantity $w(\rho,h)$ for the case when $h\le-2$. Then we shift the line of integration to the line $\Re s =-\frac12 h L$, which is positive (as $h<0$), whence the integrand is analytic between the old and new lines of integration, and in view of the fast decrease of the integrand towards $i\infty$, the formula $w(\rho,h)=\frac{1}{2\pi i} \int_{(-\frac12 h L)} \frac{\zsr}{s} \exp\left(s^2/L+h s\right)~ds$ is justified. Therefore, taking into account \eqref{zsintheright}, we obtain the estimate
$$
|w(\rho,h)|\le \frac{1}{2\pi} (A+\kappa) \frac{\beta-hL/2}{\beta-hL/2-1} \frac{2}{L} \int_0^\infty e^{\frac14 h^2 L -t^2/L -\frac12 h^2 L} dt < (A+\kappa) e^{-\frac14 h^2 L}.
$$
So we assume now $X\ge e^2 Y$, i.e., $\log X \ge \lambda +2$, and consider $g \in \G$ with $|g|\ge X$, i.e., $h=\lambda-\log|g|\le -2$. The part with $|g| \ge X$ of the above sum \eqref{Iwithw} can be estimated as
\begin{align}\label{Irhilargeg}
\bigg| \sum_{|g|\ge X} \frac{\mu(g)}{|g|^{\rho}} w(\rho, \lambda-\log|g|) \bigg| &\le \sum_{|g|\ge X} (A+\kappa) e^{-\frac14 (\log|g|-\lambda)^2 L}
=
(A+\kappa) \int_X^\infty e^{-\frac14 (\log x-\lambda)^2 L} d\N(x).
\end{align}
For the integral partial integration yields
\begin{align}\label{Irhilargeg}
\notag & = \left[ e^{-\frac14 (\log x-\lambda)^2 L} \N(x) \right]_X^\infty + \int_X^\infty \frac{L}{2} (\log x-\lambda) \frac{1}{x} e^{-\frac14 (\log x-\lambda)^2 L} \N(x) dx
\\ \notag & \le (A+\kappa) \int_X^\infty \frac{L}{2} (\log x-\lambda) e^{-\frac14 (\log x-\lambda)^2 L} dx
\\ & = (A+\kappa) \int_{\log(X/Y)}^\infty \frac{L}{2} ~y~ e^{-\frac14 y^2 L} e^y dy
\\ \notag & \le (A+\kappa) \int_{\log(X/Y)}^\infty 2 \left(\frac{L}{2} y -1\right) e^{-\frac{L}{4} y^2+y} dy
\\ &= 2(A+\kappa) e^{-\frac{L}{4} \log^2(X/Y)+\log(X/Y)} \le 2 (A+\kappa) e^{2-L} = 2e^2 (A+\kappa) \frac{1}{T},\notag
\end{align}
because the expression in the exponent is a decreasing function in $\log(X/Y)\ge 2$. Thus, we are led to
\begin{equation}\label{Irhoforlargegest}
\bigg| \sum_{|g|\ge X} \frac{\mu(g)}{|g|^{\rho}} w(\rho, \lambda-\log|g|) \bigg| \le 2e^2 (A+\kappa)^2 \frac{1}{T}<0.1,
\end{equation}
if $T \ge T_1:=200 (A+\kappa)^2$, say. Altogether from \eqref{Irhofirstform} and \eqref{Irhoforlargegest} we get for $T\ge T_1$ and $X\ge e^2 Y$ the estimate
\begin{equation}\label{Irhosmallpart}
\bigg| I(\rho,X) -1 \bigg| \le 0.2, \qquad \textrm{where}\qquad I(\rho,X):=\sum_{|g|\le X} \frac{\mu(g)}{|g|^{\rho}} w(\rho, \lambda-\log|g|).
\end{equation}

\subsection{The terms with large (positive) $h$ in $w(\rho,h)$}

The evaluation of the terms in $I(\rho,X)$ (i.e., the ones with $h\ge -2$) will be done differently,
not melting into a comprised quantity $h$ the two exponents $\lambda$ and $-\log|g|$, but handling them separately. Recalling that $\xi \in (\theta,\si)$, and $\eta:=1-\si>0$, $\de:=\si-\xi>0$, so that $1-\xi=\eta+\de$, the line of integration in $w(\rho,h)$ is moved from the line with $\Re s=3$ to the left, to the line $\Re s=\xi-\beta$. In view of the fast decrease of the integrand towards $i\infty$, transferring the line is without any problem, but the strip $\xi-\beta<\Re s<3$ between the old and new vertical lines of integration may contain singularities of the integrand.

Namely, the function $\zsr$ has a first order pole singularity at $s=1-\rho$ whose real part is $1-\beta$, whence it is in the strip $\xi-\beta<\Re s <3$. As the residuum of $\zs$ at $s=1$ is $\kappa$, at $s+\rho=1$, i.e., at $s=1-\rho$ the residuum of the whole integrand amounts to $\kappa \cdot \frac{1}{1-\rho} \exp((1-\rho)^2/L+(\lambda-\log|g|) (1-\rho))$ with absolute value not exceeding $\frac{\kappa }{\gamma} e^{1/L+(\lambda-\log|g|) (1-\beta) -\gamma^2/L}$. From here on, we will also use that $\Re \rho=\beta> \sigma=1-\eta$, (with $0<\eta \le1/12$) and $\Im \rho=\gamma \ge 10 L$. Then the above residuum is at most
$$
\frac{\kappa }{\gamma} e^{1/L +(\lambda-\log|g|+2) \eta - 2 (1-\beta) -\gamma^2/L} \le \frac{\kappa e^{1/L+2\eta}}{\gamma |g|^\eta} Y^\eta e^{-10\gamma} < \frac{2 \kappa }{\gamma |g|^\eta} Y^\eta e^{-10\gamma} .
$$

Further, when $s=0$, in principle there is another singularity of the integrand. Therefore, here we finally \emph{assume that $\rho$ is a zero of $\zs$},
so that the vanishing of $\zsr$ extinguishes the first order pole of the kernel $\frac{1}{s}\exp(s^2/L+hs)$ at $s=0$, making the point a removable singularity, harmless for the transfer of the line of integration. In sum, we take $\rho=\beta+i\gamma \in \Z_{+}(\sigma,10L,T)$, a zero with real part at least $\sigma$ and imaginary part between $10L$ and $T$.

Let us add up all the contributions arising from the residue of the translated zeta function $\zsr$ at its pole at $s=1-\rho$. We get, using also $\beta>\si$,
i.e., $\beta+\eta>1$, the estimate
\begin{align}\label{eq:zsrpole}
\bigg| \sum_{|g|\le X}\frac{\mu(g)}{|g|^\rho}~ &\textrm{Res}\left[\frac{\zsr}{s} \exp(s^2/L+(\lambda-\log|g|) s) ~ ;~ s=1-\rho \right] \bigg|
\\& \le \frac{2 \kappa  Y^\eta }{\gamma e^{10\gamma}} \int_{1}^X \frac{1}{t^{\beta+\eta}} d \N(t)
\le \frac{2 \kappa Y^\eta }{\gamma e^{10\gamma}} \int_{1}^X \frac{1}{t} d \N(t)
= \frac{2 \kappa Y^\eta }{\gamma e^{10\gamma}}\zeta_X(1) \le \frac{(A+\kappa)^2}{1-\theta} \frac{Y^\eta \log X}{L T^{100}}, \notag
\end{align}
applying Lemma \ref{l:zetaxs}, \eqref{zxesti}, the middle line in the last estimate.

The essential part of $I(\rho,X)$ is to be the integral on the changed path $\Re s = \xi-\beta \in (-1,0)$. More precisely, we will find that the essential contribution of the integral over the whole line comes from the part where $\Im s=t \in [-2L,2L]$. Denoting $D:=D(A,\kappa,\theta,\xi):=\frac{(A+\kappa)(1-\theta)}{(1-\xi)(\xi-\theta)}$, we get from Corollary \ref{c:Mesti} for any $h>-2$
\begin{align*}
\bigg| \frac{1}{2\pi i} \int_{\Re s=\xi-\beta, \atop |\Im s|\ge 2L} \frac{\zsr}{s} \exp\left(s^2/L+h s\right)~ds \bigg| & \le \frac{e^{1/L+(\xi-\beta)h}}{2\pi} 2\int_{2L}^\infty  \frac{M(\xi,t+\gamma)}{t} e^{-t^2/L} dt
\\ & \le \frac{e^{1/L+(\xi-\beta)h}}{\pi} \int_{2L}^\infty \frac{2D ~(t+\gamma)^{\frac{1-\xi}{1-\theta}} }{t} e^{-t^2/L} dt
\\ & \le \frac{2 e^{1/L+1} D} {\pi} \int_{2L}^\infty (2\max(t,\gamma))^{\frac{1-\xi}{1-\theta}} \frac{1}{t} e^{-t^2/L} dt
\\ & \le \frac{2 \cdot e^{1.1} \cdot \sqrt{2}}{\pi} D \left\{\int_{2L}^\gamma\gamma^{\frac{1-\xi}{1-\theta}} \frac{1}{t} e^{-t^2/L} dt +\int_\gamma^\infty t^{\frac{1-\xi}{1-\theta}} \frac{1}{t} e^{-t^2/L} dt\right\},
\end{align*}
using also that $\xi>\frac{1+\theta}{2}\ge 1/2$, $h\ge -2$, $\beta<1$ and $L\ge 10$ in view of $T\ge e^{10}$. In the first integral we can estimate $\gamma^{\frac{1-\xi}{1-\theta}} \le \gamma \le T$ by assumption, and as $1/t < 2t/L$, we have
$$
\int_{2L}^{\gamma} \gamma^{\frac{1-\xi}{1-\theta}} \frac{1}{t} e^{-t^2/L}dt \le \int_{2L}^{\infty} T \frac{2t}{L} e^{-t^2/L}\, dt = T^{-3}.
$$
For the second integral we use that $10L \le \gamma \le t \le T$ and calculate as follows:
$$
\int_\gamma^\infty t^{\frac{1-\xi}{1-\theta}} \frac{1}{t} e^{-t^2/L} dt \le \int_{10L}^\infty t \frac{1}{10 L} e^{-t^2/L}\, dt = \frac{1}{20} e^{-100L} \le T^{-100}.
$$
The above estimates thus lead to
$$
\bigg| \frac{1}{2\pi i} \int_{\Re s=\xi-\beta, |\Im s|\ge 2L} \frac{\zsr}{s} \exp\left(s^2/L+h s\right)\, ds \bigg| \le 3D~ \frac{1}{T^3}.
$$
Next, similarly to \eqref{eq:zsrpole} we add up all these contributions for all $|g|\le X$ in the sum for $I(\rho,X)$. With a reference to the first form in the first line of \eqref{zetaXestimate} and using again $\beta>\sigma$ we get
\begin{align}\label{eq:farintegral}
\bigg| \sum_{|g|\le X} \frac{\mu(g)}{|g|^\rho} & \frac{1}{2\pi i} \int_{\Re s=\xi-\beta, |\Im s|\ge 2L} \frac{\zsr}{s} \exp\left(s^2/L+h s\right)~ds \bigg|
\\ & \notag \le \sum_{|g|\le X} \frac{|\mu(g)|}{|g|^\beta} 3D \frac{1}{T^3}
= 3D\frac{1}{T^3}\zeta_X(\beta)
\le 3D \frac{1}{T^3} \left(\kappa X^{1-\beta} +
\frac{A}{\beta-\theta} \right)\le \frac{3 D~(A+\kappa)}{\si-\theta} \frac{X^\eta }{T^3}.
\end{align}
So assuming now the light condition that $\log X <T^{97}$ and collecting \eqref{eq:zsrpole} and \eqref{eq:farintegral} we find that the residues and the integrals in $w(\rho,\lambda-\log|g|)$ restricted to $t=\Im s \not\in [-2L,2L]$ contribute at most
$$
\frac{3~(A+\kappa)^2(1-\theta)}{(1-\xi)(\xi-\theta)(\si-\theta)} \frac{X^\eta }{T^3} \le 0.1,
$$
provided that $T\ge T_2(\xi,\si,X):= \max\left(\log^{1/97}X;\left(\frac{30~(A+\kappa)^2(1-\theta)}{(1-\xi)(\xi-\theta)(\si-\theta)}  \right)^{1/3} X^{\eta/3}\right)$.
In other words, if we introduce the notation
\begin{align}\label{eq:IrhoXL} \notag
I(\rho,X,L)& := \sum_{|g|\le X} \frac{\mu(g)}{|g|^{\rho}} \frac{1}{2\pi} \int_{-2L}^{2L} \frac{\zeta(\xi+i(t+\gamma))}{\xi-\beta+it} e^{(\xi-\beta+it)^2/L+(\lambda-\log|g|)(\xi-\beta+it)}dt
\\& = \frac{1}{2\pi} \int_{-2L}^{2L} \sum_{|g|\le X} \frac{\mu(g)}{|g|^{\xi+i(\gamma+t)}} \frac{\zeta(\xi+i(t+\gamma))}{\xi-\beta+it} e^{\frac{(\xi-\beta)^2-t^2+2it(\xi-\beta)}{L}+\lambda(\xi-\beta+it)} dt ,
\end{align}
then we have already  derived the following estimate.

\begin{lemma}\label{l:IrhoXL}
For the quantity defined in \eqref{eq:IrhoXL} we have for all $T \ge T_1, T_2$ the estimate
$$
\left| I(\rho,X,L) -1\right| \le 0.3.
$$
\end{lemma}

\subsection{An application of Halász' method for the upper estimation of $I(\rho,X,L)$}

Next we compute an upper estimation for the quantity $I(\rho,X,L)$. Since $|\gamma|\le T$ and $t\in[-2L,2L]$, in the integral we can estimate $\zeta(\xi+i(t+\gamma))$ by $M(\xi,2T)$. Therefore, writing $M:=M(\xi,2T)$ for short,
\begin{align*}
\int_{-2L}^{2L} \bigg|\zeta(\xi+i(t+\gamma)) &
\frac{\exp((\xi-\beta)^2/L-t^2/L+2it(\xi-\beta)/L+\lambda(\xi-\beta+it))}{\xi-\beta+it} \bigg| dt
\\& \le M e^{(\xi-\beta)^2/L+\lambda(\xi-\beta)} ~2\int_0^{2L} \frac{e^{-t^2/L}~dt}{\sqrt{(\xi-\beta)^2+t^2}}
\le 2 M e^{\de^2/L+\lambda(\xi-\si)} ~\int_0^{\infty} \frac{e^{-t^2/L}}{\max(\de,t)}~dt
\\ &\le 2 M ~1.01 ~Y^{\xi-\si} \left(\de+ \log\frac{2\sqrt{L}}{\de} + \frac{1}{e}\right) \le  4 M Y^{-\de} \log L,
\end{align*}
if we assume $T>T_3(\de):=e^{1/\de}$, whence $\log(1/\de)\le \log L$, too. Thus, we must have for all $T \ge T_1, T_2, T_3$
$$
|I(\rho,X,L)| \le \frac{1}{2\pi} \max_{-2L\le t \le 2L}  \left|\sum_{|g|\le X} \frac{\mu(g)}{|g|^{\xi+i(\gamma+t)}} \right|  \cdot 4  M Y^{-\de} \log L.
$$
However, we have already seen that $|I(\rho,X,L)|$
appearing on the left hand side is at least $0.7$. It follows that there exist some $\tau:=\tau(\rho) \in [-2L,2L]$ and a corresponding complex unit $\alpha:=\alpha(\rho,\tau):=e^{i\varphi(\rho,\tau)}$, where $\varphi:=\varphi(\rho,\tau):=\arg \left(\sum_{|g|\le X} \frac{\mu(g)}{|g|^{\xi+i(\gamma+\tau)}}\right)$, such that
\begin{equation}\label{SumXrholower}
\alpha ~\sum_{|g|\le X} \frac{\mu(g)}{|g|^{\xi+i(\gamma+\tau)}} \ge \frac{2\pi \cdot 0.7}{4} \frac{ Y^\de }{M \log L} > 1.1 \cdot \frac{ Y^\de }{M \log L} \qquad (M:=M(\xi,2T)).
\end{equation}

Now let us take a subset $\SSS$ of $\Z_{+}(\si;10L,T)$ of Beurling $\zeta$ zeros, numbered as $\SSS=\{ \rho_k=\beta_k+i\gamma_k ~:~k=1,\ldots,K\}$, so that $\#\SSS=K$. The corresponding parameter values from the above considerations will be denoted as $\tau_k:=\tau(\rho_k)$ and $\alpha_k:=\alpha(\rho_k,\tau_k)$. Following Halász, we sum up the inequalities \eqref{SumXrholower} for all $\rho_k$ ($k=1,\ldots,K$), square both sides, and apply the Cauchy--Schwartz inequality. This yields
\begin{align}\label{Keyinequalities}
1.2 \cdot K^2 & \cdot \notag \frac{ Y^{2\de }}{M^2 \log^2 L}
\le \left( \sum_{k=1}^K \alpha_k \sum_{|g|\le X} \frac{\mu(g)}{|g|^{\xi+i(\gamma_k+\tau_k)}} \right)^2
= \left( \sum_{|g|\le X} \frac{\mu(g)}{\sqrt{|g|}} \sum_{k=1}^K \frac{\alpha_k }{|g|^{\xi-1/2+i\ok}} \right)^2
\\ \notag & \le   \sum_{|g|\le X} \frac{\mu^2(g)}{|g|} \sum_{|g|\le X} \left|\sum_{k=1}^K \frac{\alpha_k}{|g|^{\xi-1/2+i\ok}} \right|^2
\le \zeta_X(1) \left( \sum_{|g|\le X} \sum_{k=1}^K \sum_{j=1}^K \frac{\alpha_k \overline{\alpha_j}}{|g|^{2\xi-1+i\ok-i\oj}} \right)
\\ \notag & = \zeta_X(1) \left( \sum_{k=1}^K \sum_{j=1, j\ne k}^K \alpha_k \overline{\alpha_j} \sum_{|g|\le X} \frac{1}{|g|^{2\xi-1+i\ok-i\oj}} + \sum_{k=1}^K \sum_{|g|\le X} \frac{1}{|g|^{2\xi-1}} \right)
\\ & = \zeta_X(1)\left(\sum_{k=1}^K \sum_{j=1, j\ne k}^K \alpha_k \overline{\alpha_j}  ~\zeta_X(2\xi-1+i(\ok-\oj))+ K \zeta_X(2\xi-1) \right)
\notag
\\ & \le \frac{A+\kappa}{1-\theta} \log X \left(\sum_{k=1}^K \sum_{j=1, j\ne k}^K \left|\zeta_X(2\xi-1+i(\ok-\oj))\right|+ K (A+\kappa)\frac{X^{2-2\xi}}{2\xi-1-\theta} \log X\right),
\end{align}
taking into account the middle line of \eqref{zxesti} from Lemma \ref{l:zetaxs} and also \eqref{zetaXestimate} of Lemma \ref{l:zetasigma} in the last step.

The interesting terms come from the double sum. We assume, as it will be justified below, that the different $\ok$ are at least $2$ apart, so that each $\zeta_X(2\xi-1+i(\ok-\oj))$ term in this double sum can be estimated by \eqref{zetaXsecond} from Lemma \ref{l:zzzz}. If we choose the log-free part from the first minimum expression in \eqref{zetaXsecond}, then  the double sum is seen not exceeding
\begin{align}\label{Doublesumwithlemma}
\frac{(2A+\kappa)}{(2\xi-1-\theta)} & \sum_{k=1}^K \sum_{j=1, j\ne k}^K \left\{ \frac{1-\theta}{2-2\xi}|\ok-\oj|^{\frac{2-2\xi}{1-\theta}} +
   \left( \frac{X^{2-2\xi}}{|\ok-\oj|}+ \frac{|\ok-\oj|}{X^{2\xi-1-\theta}} \right)\right\}
\\ & \le \frac{(2A+\kappa)(1-\theta)}{(2\xi-1-\theta)(2-2\xi)} K^2 \left( T^{\frac{2-2\xi}{1-\theta}}+  \frac{T}{X^{2\xi-1-\theta}} \right) + \frac{2A+\kappa}{2\xi-1-\theta}\sum_{k=1}^K \sum_{j=1, j\ne k}^K \frac{X^{2-2\xi}}{|\ok-\oj|}. \notag
\end{align}

Assume, as we may, that the $\ok$ are indexed according to increasing magnitude, and take the (minimal) separation between any two as a parameter $Q$. (We will see in a moment that the concrete set $\SSS$ admits such a separation.) Then the separation between $\ok$ and $\oj$ is at least $|k-j|Q$. Using this, we are led to
$$
\sum_{k=1}^K \sum_{j=1, j\ne k}^K \frac{1}{|\ok-\oj|} \le 2~ \sum_{k=1}^K \sum_{j=k+1}^K \frac{1}{(j-k)Q} < \frac{2}{Q} K(1+\log K).
$$
Assume further (as we will see in a moment from the concrete definition of the set $\SSS$ right below) that $K\le T/e$. Then $(1+\log K)\le \log T=L$, and we are led to
\begin{equation}\label{reciproksum}
\sum_{k=1}^K \sum_{j=1, j\ne k}^K  \frac{X^{2-2\xi} }{|\ok-\oj|} \le K \frac{2L}{Q} X^{2-2\xi}.
\end{equation}

At this point we finally specify the set $\SSS$. As said, it will be a subset of $\Z_{+}(\si,10L,T)$
chosen with a maximal number of elements under the condition that the imaginary parts $\gamma_k=\Im \rho_k$ are at least $10L$ apart. In other words, take $\rho_1$ from $\Z_{+}(\si,10L,T)$ with minimal possible imaginary part $\gamma_1$, and then inductively choose $\rho_{k+1}$ with minimal imaginary part $\gamma_{k+1}$ not below $\gamma_k+10L$ once $\rho_k$ is already selected. The construction then terminates after at most $T/(10L)$ steps, justifying our assumption that $K\le T/e$. Also, recalling that $|\ok-\gamma_k|\le 2L$, we find that the separation between the elements of the sequence $(\ok)$ is at least $Q=6L$. (Note also that the indexing is in the natural, increasing order of the $\ok$.)

Winding up the estimations done, from \eqref{Keyinequalities}, \eqref{Doublesumwithlemma} and \eqref{reciproksum} after a cancellation by $K$ we are led to
\begin{align}\label{Keysimplified}
1.2 \cdot K  \cdot \frac{Y^{2\de }}{M^2 \log^2 L}
& \le \frac{(A+\kappa)^2}{(1-\theta)(2\xi-1-\theta)} \log^2 X ~X^{2-2\xi}+ \frac{A+\kappa}{1-\theta} \log X \frac{2A+\kappa}{2\xi-1-\theta} \frac{1}{3} X^{2-2\xi} \notag
\\& \qquad \qquad + \frac{(2A+\kappa)(A+\kappa)}{(2-2\xi)(2\xi-1-\theta)} \log X ~K ~ \left( T^{\frac{2-2\xi}{1-\theta}} +  \frac{T}{X^{2\xi-1-\theta}} \right)
\\
& \le \frac{1.1 \cdot (A+\kappa)^2}{(1-\theta)(2\xi-1-\theta)} \log^2 X ~X^{2-2\xi}
+\frac{(A+\kappa)^2}{(1-\xi)(2\xi-1-\theta)} \log X ~\left( T^{\frac{2-2\xi}{1-\theta}} +  \frac{T}{X^{2\xi-1-\theta}} \right)~K. \notag
\end{align}

\subsection{Estimation of $K$}\label{sec:K}

Here we specify our parameters. We take $X:=T^{\frac{3.5}{1-\theta}}$ 
and $\de:=1.5 \eta$ so that $1-\xi= \de+\eta= 2.5 \eta$ and the condition $\xi>\frac{1+\theta}{2}$ will be met as long as $\eta<\frac15 (1-\theta)$.

\begin{lemma}\label{l:K-estimate} Let $T_4:=T_4(A,\kappa,\theta,\eta):=(A+\kappa)^{\frac{40(1-\theta)}{\eta}}$ and $T_5:=T_5(\kappa,\theta,\eta):=\exp\left(\frac{25(1-\theta)}{\eta}\log\frac{1}{\eta}\right)$. With the above parameter choices and for $T \ge \max(T_4,T_5)$ we have
\begin{equation}\label{Kfinal}
K \le 90 \frac{(A+\kappa)^4}{(1-\theta)^2\eta^4}~L^3~T^{\frac{12 \eta}{1-\theta}}.
\end{equation}
\end{lemma}

\begin{proof}
Recall that at the outset we restricted the argument to $\eta\le (1-\theta)/12$ (as otherwise Lemma \ref{l:Littlewood} already furnished the assertion of the theorem). In view of this stronger condition the inequalities
\begin{equation}\label{xietadelta}
\xi-\theta = 1-2.5\eta-\theta \ge \frac{19}{24} (1-\theta)  \quad \textrm{and} \quad 2\xi-1-\theta = 1-5\eta-\theta \ge \frac{7}{12} (1-\theta)
\end{equation}
are obtained easily, so that \eqref{Keysimplified} entails
\begin{align*}
1.2\cdot K\frac{T^{\frac{7\de}{1-\theta}}e^{-4\de}}{M^2 \log^2 L}
& \le \frac{1.1 \cdot 12 \cdot (A+\kappa)^2}{7(1-\theta)^4} 3.5^2~\log^2 T ~T^{\frac{7(1-\xi)}{1-\theta}}
+\frac{12~(A+\kappa)^2}{2.5\eta ~7~(1-\theta)^2} 3.5 \log T ~\left(T^{\frac{2(1-\xi)}{1-\theta}} + T^{\frac{(1-\theta)-3.5(1-5\eta-\theta)}{1-\theta}} \right)~K \notag
\end{align*}
or, cancelling  $T^{\frac{7\de}{1-\theta}}$ but otherwise equivalently,
\begin{align}
1.2\cdot K\frac{e^{-6\eta}}{M^2 \log^2 L}
& \le \frac{23.1 \cdot (A+\kappa)^2}{(1-\theta)^4}~L^2 ~T^{\frac{7\eta}{1-\theta}}
+\frac{12~(A+\kappa)^2}{5\eta (1-\theta)^2} L ~\left(T^{\frac{-5.5\eta}{1-\theta}} + T^{\frac{17.5\eta-2.5(1-\theta)}{1-\theta}} \right)~K.
\end{align}
Since $\eta<(1-\theta)/12$ and $17.5\eta<18\eta \le 1.5(1-\theta)$, here on the right hand side we have $T^{\frac{17.5\eta-2.5(1-\theta)}{1-\theta}} \le T^{-1}\le T^{\frac{-5.5\eta}{1-\theta}} T^{-1/2}$ so in view of $T\ge 200>14^2$ also $T^{\frac{17.5\eta-2.5(1-\theta)}{1-\theta}} \le \frac{1}{14}T^{\frac{-5.5\eta}{1-\theta}}$.
Using this and on the left hand side also $6\eta<(1-\theta)/2\le 1/2$ and $\log^2 L \le (2/e)^2 L$, we get
$$
1.34 \cdot \frac{K}{M^2 L} < 1.2\cdot K\frac{e^{-1/2} e^2/4}{M^2 L} \le 1.2\cdot K \frac{e^{-6\eta}}{M^2 \log^2 L} \le \frac{23.1 \cdot (A+\kappa)^2}{(1-\theta)^4}~L^2 ~T^{\frac{7\eta}{1-\theta}}
+\frac{18~(A+\kappa)^2}{7\eta (1-\theta)^2} L ~T^{\frac{-5.5\eta}{1-\theta}}~K,
$$
or in other words
\begin{align}\label{Keycompressed}
1.34 ~K \le \frac{23.1 \cdot (A+\kappa)^2}{(1-\theta)^4}~L^3 M^2 ~T^{\frac{7\eta}{1-\theta}}
+\frac{18~(A+\kappa)^2}{7\eta (1-\theta)^2} L^2 M^2 ~T^{\frac{-5.5\eta}{1-\theta}}~K .
\end{align}
Recall that $M=M(\xi,2T)$. A reference to Corollary \ref{c:Mesti} furnishes from here (also using $2^{\frac{2(1-\xi)}{1-\theta}} =2^{\frac{5\eta}{1-\theta}} \le \sqrt{2}$)
$$
1.34 ~K \le \frac{23.1 \cdot (A+\kappa)^2}{(1-\theta)^2}~L^3 \sqrt{2} \frac{(2A+\kappa)^2}{(1-\xi)^2(\xi-\theta)^2}~T^{\frac{12 \eta}{1-\theta}}
+\frac{18~(A+\kappa)^2}{7\eta } L^2 \sqrt{2} \frac{(2A+\kappa)^2}{(1-\xi)^2(\xi-\theta)^2} ~T^{\frac{-0.5\eta}{1-\theta}}~K,
$$
so that according to \eqref{xietadelta} we get
$$
1.34 K \le \frac{23.1 \cdot \sqrt{2} (A+\kappa)^4}{(1-\theta)^2}~L^3  \frac{16\cdot 12^2}{25 \eta^2 7^2 \eta^2}~T^{\frac{12 \eta}{1-\theta}}
+\frac{18~\sqrt{2} ~(A+\kappa)^2}{7\eta } L^2 \frac{16 (A+\kappa)^2}{25 \eta^2 7^2 \eta^2} ~T^{\frac{-0.5\eta}{1-\theta}}~K.
$$
For the second term we apply $L=\log T = \frac{5(1-\theta}{\eta} \log \left( T^{\frac{\eta}{5(1-\theta)}} \right) \le \frac{10}{e} \frac{1-\theta}{\eta} T^{\frac{0.1 \eta}{1-\theta}}$ and get
$$
1.34 K \le \frac{23.1 \cdot \sqrt{2} (A+\kappa)^4}{(1-\theta)^2}~L^3  \frac{16\cdot 12^2}{35^2 \eta^4 }~T^{\frac{12 \eta}{1-\theta}}
+\frac{18~\sqrt{2} ~(A+\kappa)^4}{7\eta } \frac{1600 }{(35e)^2 \eta^4 } ~T^{\frac{-0.3\eta}{1-\theta}}~K,
$$
or, computing the constants,
\begin{equation}\label{computingconstants}
1.34 K \le 62 \frac{(A+\kappa)^4}{(1-\theta)^2\eta^4}~L^3~T^{\frac{12 \eta}{1-\theta}}
+\frac{0.65~(A+\kappa)^4}{\eta^5} ~T^{\frac{-0.3\eta}{1-\theta}}~K.
\end{equation}
As $T\ge \max(T_4,T_5)$ we get $(A+\kappa)^4 \le T^{\frac{0.1\eta}{1-\theta}}$ and also $\frac{1}{\eta} \le T^{\frac{\eta}{25(1-\theta)}}$, so that on the right hand side the last term is estimated by a constant times $K$, more precisely
$$
1.34 K \le 62 \frac{(A+\kappa)^4}{(1-\theta)^2\eta^4}~L^3~T^{\frac{12 \eta}{1-\theta}} + 0.65 ~K,
$$
whence $0.69 K \le 62 \frac{(A+\kappa)^4}{(1-\theta)^2\eta^4}~L^3~T^{\frac{12 \eta}{1-\theta}}$ and
Lemma \ref{l:K-estimate} follows.
\end{proof}

\subsection{Conclusion of the proof of Theorem \ref{th:NewDensity}}

It remains to compare $K$ with $N(\si,T)$.
By construction, the union of the intervals $[\gamma_k,\gamma_k+10L]$ covers $\Z_{+}(\si,10L,T)$. Therefore,
Lemmas \ref{l:Littlewood} and \ref{c:zerosinrange} furnish
\begin{align}\label{Nfirstestimate}
N(\si,T)& \le N(\si,10L)+ 2 \left( \sum_{k=1}^{K-1} N(\si,\gamma_k,\gamma_k+10L) + N(\si,\gamma_K,\min(\gamma_K,T)) \right) \notag
\\ & \le N(\si,10L) + 2 K \max_{10L \le R \le T} N(\si,R,\min(R+10L,T)) \notag
\\ &\le \frac{1}{\si-\theta}
\left\{\frac{1}{2} 10L \log (10L) + \left(\log\frac{(A+\kappa)^2}{\si-\theta} + 3 \right)10L\right\} \notag
\\ &\quad + 2 K\cdot \frac{1}{\si-\theta} \left\{ \frac{4}{3\pi} 10L \left(\log\left(\frac{11.4 (A+\kappa)^2}{\si-\theta}{T}\right)\right)  + \frac{16}{3}  \log\left(\frac{60 (A+\kappa)^2}{\si-\theta}{T}\right)\right\}.
\end{align}
Let us also assume $T \ge T_6:=T_6(\theta):=\frac{1}{(1-\theta)^{100}}$, say. Then $L\ge 100 \log \frac{1}{1-\theta}$. Further, $L\ge \log T_4$ entails $L\ge 480 \log(A+\kappa)$, and $T\ge T_5$ entails $L\ge 300$, whence taking into account also $\si-\theta\ge \frac{11}{12}(1-\theta)$ we obtain
$$
\log\left(\frac{(A+\kappa)^2}{\si-\theta}\right) \le \log\frac{12}{11} + \frac{L}{240}+\frac{L}{100} \le L\left( \frac{1}{3300}+\frac{1}{240}+\frac{1}{100} \right) \le 0.02 L.
$$
Applying this and $\log u/u \le 1/e$ or $\log u/u \le \log u_0/u_0$ whenever $u\ge u_0\ge e$
we are led to
\begin{align*}
N(\si,T) & \le \frac{12}{11(1-\theta)} \bigg\{ 50L \frac{\log 3000}{3000} L + (0.02 L +0.01 L) \cdot 10L
\\& \qquad \qquad \qquad + 2 K\left[ \frac{40}{3\pi} L \left(L + \log 13 + 0.02 L \right) + \frac{16}{3} \left(\log 60 + 0.02 L +L\right) \right]\bigg\}
\\& \le \frac{12}{11(1-\theta)} \bigg\{ 0.5 L^2 + 8.6 K\cdot L \left(L + 0.01 L + 0.02 L \right) + 10.8 K \left(4.1 + 1.02 L \right)\bigg\}
\\ & \le \frac{1}{1-\theta} \bigg\{ L^2 + 10 K\cdot L^2 + 13 K L \bigg\} \le \frac{10.5}{1-\theta} (K+1) L^2.
\end{align*}
Using \eqref{Kfinal} in this last estimate we obtain \eqref{densityresult}. Therefore, the theorem is proved as soon as we check that all our conditions are met. The assumptions were always of the form $T\ge T_k$ with some explicit $T_k$, except for the condition about $T_2$, where we assumed $T\ge T_2(\xi,\si,X):= \max\left(\log^{1/97}X;\left(\frac{30~(A+\kappa)^2(1-\theta)}{(1-\xi)(\xi-\theta)(\si-\theta)}  \right)^{1/3} X^{\eta/3}\right)$. The first part, $T^{97} \ge \log X= \frac{3.5}{1-\theta} \log T$ is obvious. Also the constant part of the second condition is easy, because $T\ge T_1=200 (A+\kappa)^2$ according to the first assumption with $T_1$, while $\frac{1-\theta}{(1-\xi)(\xi-\theta)(\si-\theta)} \le \frac{1-\theta}{2.5 \eta \cdot\frac{19}{24}(1-\theta) \cdot  \frac{11}{12}(1-\theta)} \le \frac{0.3(1-\theta)}{\eta(1-\theta)^2}=0.3 \frac{1-\theta}{\eta } \cdot \frac{1}{(1-\theta)^2} \le \exp\left(\frac{0.3(1-\theta)}{\eta }\right) \frac{1}{(1-\theta)^2} \le T_5^{0.3/25} \cdot T_6^{1/50} \le T$ in view of \eqref{xietadelta} and $\si-\theta=1-\eta-\theta \ge \frac{11}{12}(1-\theta)$. That is, we surely have $T^2\ge \frac{30~(A+\kappa)^2(1-\theta)}{(1-\xi)(\xi-\theta)(\si-\theta)}$. What remains to be checked is if we have $T\ge X^{\eta}$, too. However, $\eta<(1-\theta)/12$, whence $X^{\eta}\le T^{\frac{3.5 \eta}{1-\theta}} \le T^{1/3}$, and the condition is met.

This concludes the proof of Theorem \ref{th:NewDensity} with $T_0:=\max(T_1,T_3,T_4,T_5,T_6)$.

\section{Concluding remarks}\label{sec:concludingremarks}

In this section we attempt to give an overview of some of the current fast development of Beurling's theory connected with the emerging of density type estimates and with an emphasis on our own research interest in the matter. The results explained in this section are very fresh, but they can already be read in other recent papers with full proofs, hence we do not give proofs here. The open problems we mention here seem to be within reach by now, but they are of varying degree of difficulty.

\subsection{Some consequences of the extended effective Carlson-type density estimate}\label{sec:consequences}

As said, our detailed study of the distribution of zeroes of the Beurling zeta function was motivated by the goal to investigate the questions of Littlewood and Ingham---together with the sharpness of the obtained results---in the Beurling context. Satisfactorily precise results on these questions were obtained in \cite{Rev-One} and \cite{Rev-Many}. In the work on the Littlewood question in \cite{Rev-One}, we referred to the density theorem only in the heuristical considerations on which our construction for the proof of sharpness was based. However, in the original version of \cite{Rev-Many} we needed to use Theorem \ref{th:density} heavily in the arguments, hence our results there were restricted to Beurling systems satisfying Conditions B and G of this density result. Moreover, the constants depended also on these conditions and the somewhat implicit handling of them. During the revision process of \cite{Rev-Many}, however, we could replace this use of Theorem \ref{th:density} by a reference to Theorem \ref{th:NewDensity}, thus obtaining stronger forms of the results. These we will recall below, too.

Here we made every effort to handle the dependence of our density estimates on Axiom A and the natural parameters of the problem (basically, the quantity $\eta:=1-\si$) fully explicitly in order to guarantee that the dependence on the order and oscillation results of \cite{Rev-Many} can be made explicit, too. The arguments in this paper somewhat suffer from the clumsy explicit calculations with all the arising constants, while the result was not fully optimized\footnote{Nevertheless, our calculations seemed to suggest that the optimization of the exponent could not bring down $12$, that is $\frac{12(1-\si)}{1-\theta}$ in the exponent, too much--certainly not below 11. Our final choice was therefore to get a round, integer exponent with a somewhat less messy calculus.} regarding neither the $\log T$ power nor the exponent of $T$. One may also note that the dependence of the constant in \eqref{densityresult} on the value of $\eta$ can be eliminated by noting that the result holds true trivially if $\si$ is so large that $\zs$ has no zero at all in the rectangle. This is certainly true (and nothing more can in general be expected in view of \cite{DMV} !) if $\si>1-c/\log T$, that is, if $1/\eta > (1/c) \log T$. Using that provides the corollary that $N(\si,T)\le 1000 \frac{(A+\kappa)^4}{(1-\theta)^3} \frac{1}{c^4} T^{\frac{12}{1-\theta}(1-\si)} \log^9 T$, always. However, the value of $c$---even if it is plausible that it could be expressed explicitly by means of the main parameters $A, \kappa$ and $\theta$ of Axiom A---has not been estimated yet. So in order to preserve the effective nature of our estimate we opted for saving its current form in the main result.

Below we give the consequent effective and generalized formulations of our main results
from \cite{Rev-Many} that we could obtain by replacing the original use of Theorem \ref{th:density} by the new Theorem \ref{th:NewDensity}.

The first corollary answers Ingham's question in the Beurling context. We recall the setup of Ingham.
Denote by $\eta(t):(0,\infty)\to (0,1/2)$ a nonincreasing function and consider the domain
\begin{equation}\label{eq:etazerofree}
\DD(\eta):=\{ s=\sigma+it \in\CC~:~ \sigma>1-\eta(t),~ t>0\}.
\end{equation}
Following Ingham \cite{Ingham} consider also the derived function (the \emph{Legendre transform} of $\eta$ in logarithmic variables)
\begin{equation}\label{omegadef}
\omega_{\eta}(x):= \inf_{y>1} \left(\eta(y)\log x+\log y\right).
\end{equation}
Then we have the following result, see Corollary 10 from \cite{Rev-Many}.
\begin{theorem}\label{th:upperestimation} Let $\G$ be an arithmetical semigroup satisfying Axiom A. Then we have for any $\ve>0$ and any sufficiently large $x>x_0(\ve,A,\kappa,\theta)$ the estimate
\begin{equation}\label{eq:uppertheorem}
\Delta_\GG(x) \le A_3(\ve,A,\kappa,\theta) x\exp(-(1-\ve)\omega_\eta(x)),
\end{equation}
where both $x_0$ and $A_3$ are effectively computable constants depending on the given parameters only.
\end{theorem}
Actually, sharper results with an $\omega$-function directly derived from the set of zeros (and not depending on a domain boundary function $\eta(t)$) hold also true, see in particular Theorem 10 in \cite{Rev-Many}. Also, other PNT-and $\zs$-related quantities are estimated similarly in \cite{Rev-Many}. However, we refrain from the technicalities necessary in formulating them here.

Finally, we address the sharpness of the above Ingham-type result, i.e., a generalization of Pintz' oscillation theorem \cite{Pintz2}. We present here Theorem 6 from \cite{Rev-Many}.
\begin{theorem}\label{th:Beurlingdomainosci}
Let $\G$ be an arithmetical semigroup satisfying Axiom A, and consider a function $\eta(t)$ which is convex in logarithmic variables (i.e., $\eta(e^v)$ is convex). Further, consider the conjugate function $\omega_\eta$ defined above in \eqref{omegadef}. Assume that there are infinitely many zeroes of $\zs$ within the domain \eqref{eq:etazerofree}.

Then we have for any $\ve >0$ the oscillation estimate $\D(x)=\Omega(x\exp(-(1+\ve)\omega_\eta(x))$ with effective implied constants (depending only on $\ve, A, \kappa$ and $\theta$).
\end{theorem}
As above, similar versions with an $\eta$-independent $\omega$ function hold also true, but we skip the exact details.

It is interesting to compare this result with the one, obtained by Hilberdink and Kazyulite \cite{Hilberdink-Kazu} a few month before our work. While our result provides an essentially sharp error estimate when the boundary of the zero-free region tends to the one-line (and uses the density estimates presented here), Hilberdink and Kazyulite worked independently of any zero-density results. They derived the estimate $\Omega(x^b\exp(-c \,\omega_b(x))$ with $b:=\sup\{\Re \rho~:~\rho(\zeta)=0\}$ and
$$
\omega_b(x):=\inf_{y>1} \left((\eta(y)-(1-b))\log x+\log y\right)
$$
measuring the oscillation with a comparison of the zero-free region and its boundary line $\Re s=b$. One can say that they use the Legendre transform of $\eta_b(t):=\eta(t)-(1-b)$ for a more delicate analysis when $b<1$. Now if $b=1$, then Theorem \ref{th:Beurlingdomainosci} is sharper in that the constant is $1+\ve$, not just a constant $c$, but if $b<1$, then it only gives $\Omega(x^{b+\ve})$, which is much less precise than the Hilberdink--Kazyulite result. It would be interesting (but challenging) to determine if even in case $1/2 \le b<1$ sharp results of the order $x^b\exp(-(1\pm \ve)\omega_b(x))$ hold.

Let us offer some comments on the role of density results in these seemingly different questions of Littlewood and Ingham and the converse results showing the sharpness of these results.
As said, the original de la Vallée-Poussin argument for the classical zero-free region \eqref{classicalzerofree} and error term \eqref{classicalerrorterm} was worked out in greater generality by Landau \cite{L}, who derived the latter with any $C<\sqrt{c}$. Later Ingham generalized Landau's argument \cite{Ingham} to general zero-free regions and their corresponding conjugate functions $\omega_\eta$, but his result suffered from the same loss of precision (a factor halving the constant in the exponent, i.e., getting only $|\D(x)|\ll x\exp((\frac12-\ve)\omega(x))$). It turned out only much later \cite{Pintz2} that (some) density theorem needs to be invoked to obtain sharp conclusions in this direction (sharpness demonstrated
by the exact converse, i.e., an oscillation result).
In this regard, any density theorem with $o(1)$ exponent as $\si\to 1-$ suffices, so after Carlson's result \cite{Carlson} the necessary tools were more than available---still it took quite a while until number theorists realized what the sharp form of Landau's and Ingham's estimates would be. In fact, the realization of this possibility of sharpening the original estimates
(so, e.g., in Landau's case to get $C=2\sqrt{c}-\ve$)
was prompted by the search for the sharp converse, i.e., an oscillatory result, because until the error term estimate itself is not sharp, a precisely corresponding converse cannot be obtained either.

\subsection{Some further questions of interest}

Let us finally mention a few related questions of varying degree of difficulty that we consider interesting at this stage of development of the Beurling theory of arithmetical semigroups. As mentioned above, we would be interested to see---a possibly optimal in principle---effective zero-free region estimate of the form \eqref{classicalzerofree}. This would immediately imply by Theorem \ref{th:upperestimation} above the respective effective error term estimate in the PNT. Note that in case of the Riemann zeta function many later tricky improvements were combined to sharpen the effective zero-free region, all being started by an insightful paper of Stechkin \cite{St}, and continued, e.g., in \cite{Kadiri}  and \cite{MT-JNT}.
However, from Stechkin on, these authors capitalize on the fact that together with each root of the Riemann $\zeta$ function the symmetric (about  the critical line $\Re s=1/2$) point is a zero, too, which essentially refers to the functional equation. In case we have no such information, only the classic method of Landau and optimization of the Landau extremal problem on the respective auxiliary nonnegative trigonometric polynomials is possible. For that direction see \cite{AK} and \cite{Rev-L}.

Naturally, we do not think that the constant 12 in the exponent would be optimal. In fact, later developments improved on it already. After we finalized this paper (and it was under evaluation at a journal), Frederik Broucke and Gregory Debruyne obtained Theorem \ref{th:NewDensity} with some better exponent \cite{BrouckeDebruyne}. In particular, in the case when Condition G is also satisfied, they could also recover our earlier exponent $\frac{6-2\theta}{1-\theta}$.
The proof goes along different, more classical, lines following the large sieve type Dirichlet polynomial mean value estimate approach.
Very recently, these estimates were further improved in \cite{Broucke-new} applying a combination of several ideas and results including a nice, also very recent mean value estimate result from \cite{Broucke-Hilberdink}.

Once we have a density theorem, much of classical number theory---e.g., estimates for primes in short intervals---can possibly be extended to Beurling systems. In fact, some more assumptions are needed for that, but with extra assumptions this is indeed possible, see \cite{BrouckeDebruyne}, too.

It is an interesting question\footnote{The question was posed to us in an email by Hugh L.~Montgomery pointing out that Vinogradov himself claimed that his estimates ``in themselves'' provide the respective improved bounds, but never provided a convincing proof for this statement. The question is hard to interpret and thus to investigate (what does it mean ``in itself'' when we work with a concrete function like the Riemann zeta function), but if we fix our setup to Beurling systems, then the question is meaningful (and challenging).}
if assuming extra hypotheses such as some Vinogradov-type better estimates on $\zs$ either exclusively on $\Re s=1$ or possibly in some neighborhood of it, entail in themselves the drastic improvement known about the prime distribution (as is proved in the case of the Riemann zeta and natural prime numbers, but invoking in the proofs not only Vinogradov's estimates,
but also further extra information known only for the very case).
This assertion can only be analysed if only a generalized setup is fixed without additional special information characteristic to the natural number system and the Rieamnn zeta function. Beurling systems offer themselves as a very natural general setup for studying these issues.

Naturally, the same questions are there both for zero-free regions and also for density estimates.

\bigskip

\noindent
\hspace*{5mm}
\begin{minipage}{\textwidth}
\noindent
\hspace*{-5mm}
Szilárd Gy.{} Révész\\
HUN-REN Alfréd Rényi Institute of Mathematics\\
Reáltanoda utca 13-15\\
1053 Budapest, Hungary \\
{\tt revesz.szilard@renyi.hu}
\end{minipage}


\begin{thebibliography}{AAAA}

\bibitem{H-15}{\sc F. Al-Maamori} and {\sc T. W. Hilberdink},
An example in Beurling's theory of generalised primes.
\emph{Acta Arith.} {\bf 168} (2015), no. 4, 383--396.

\bibitem{AK} {\sc V. V. Arestov, V. P. Kondratev},  An extremal problem for nonnegative trigonometric polynomials. (Russian) \emph{Mat. Zametki} {\bf 47} (1990), 15--28, 171; translation in \emph{Math. Notes} {\bf 47} (1990), no. 1-2, 10--20.

\bibitem{Beur}
{\sc A. Beurling}, Analyse de la loi asymptotique de la
distribution des nombres premiers g\'en\'eralis\'es I. \emph{Acta
Math.} {\bf 68} (1937), 255--291.

\bibitem{Bombieri}{\sc E. Bombieri}, Density theorems for the zeta function, in: \emph{Proceedings of the Stony Brook Conference}, AMS (providence, RI, 1969), pp. 352--358.

\bibitem{Broucke-new} {\sc F. Broucke}, On zero-density estimates for Beurling zeta functions. ArXiv preprint. See at {\tt https://arxiv.org/abs/2409.10051}.

\bibitem{BrouckeDebruyne}{\sc F. Broucke, G. Debruyne}, On zero-density estimates and the PNT in short intervals for Beurling generalized numbers. \emph{Acta Arith.} {\bf 207} (2023), no. 4, 365--391.

\bibitem{BDR}{\sc F. Broucke, G. Debruyne, Sz. Gy. Révész}, Some examples of well-behaved Beurling number systems. \emph{Trans. Amer. Math. Soc.}, to appear. See also as an ArXiv preprint at {\tt https://arxiv.org/abs/2309.01567}.

\bibitem{BrouckeVindas}{\sc F. Broucke, J. Vindas}, A new generalized prime random approximation procedure and some of its applications. Math. Z. {\bf 307} (2024), no. 4, Paper No. 62, 16 pp. See also as an ArXiv preprint at {\tt https://arxiv.org/abs/arXiv:2102.08478v1}.

\bibitem{BrouckeDebruyneVindas}{\sc F. Broucke, G. Debruyne  } and {\sc J. Vindas}, Beurling integers with RH and large oscillation. \emph{Adv. Math.} {\bf 370} (2020), Article no. 107240, 38 pp.

\bibitem{Broucke-Hilberdink}{\sc F. Broucke, T. Hilberdink}, A Mean Value Theorem for general Dirichlet Series. To appear in \emph{Quart. J. Math.}. Preprint available at {\tt arXiv:2409.06301}.

\bibitem{Carlson} {\sc Carlson},
\"Uber die Nullstellen der Dirichletschen Reihen und der
Riemannschen $\zeta$-Funktion, {\it Arkiv f. Mat., Astr. och Fys.}
{\bf 15} Nr. 20 (1920).



\bibitem{DebruyneVindas-PNT}{\sc G. Debruyne } and {\sc J. Vindas},
On PNT equivalences for Beurling numbers.
\emph{Monatsh. Math.} {\bf 184} (2017), no. 3, 401--424.

\bibitem{DebruyneVindas-RT}{\sc G. Debruyne } and {\sc J. Vindas}, On general prime number theorems with remainder. \emph{Generalized functions and Fourier analysis}, 79--94, in: Oper. Theory Adv. Appl., {\bf 260}, Adv. Partial Differ. Equ. (Basel), Birkhäuser/Springer, Cham, 2017.

%

\bibitem{DSV} {\sc G. Debruyne; J.-C. Schlage-Puchta; J. Vindas},
Some examples in the theory of Beurling's generalized prime numbers.
\emph{Acta Arith.} {\bf 176} (2016), no. 2, 101--129.

\bibitem{DMV}{\sc H. G. Diamond, H. L. Montgomery  } and {\sc U. Vorhauer},
Beurling primes with large oscillation, \emph{Math. Ann.}, {\bf 334} (2006) no. 1,  1--36.

\bibitem{DZ-13-2} {\sc H. G. Diamond} and {\sc Wen-Bin Zhang},
Chebyshev bounds for Beurling numbers.
\emph{Acta Arith.} {\bf 160} (2013), no. 2, 143--157.

\bibitem{DZ-13-3} {\sc H. G. Diamond} and {\sc Wen-Bin Zhang},
Optimality of Chebyshev bounds for Beurling generalized numbers.
\emph{Acta Arith.} {\bf 160} (2013), no. 3, 259--275.

\bibitem{DZ-17} {\sc H. G. Diamond} and {\sc Wen-Bin Zhang},
Prime number theorem equivalences and non-equivalences.
\emph{Mathematika} {\bf 63} (2017), no. 3, 852--862.

\bibitem{DZ-16} {\sc H. G. Diamond} and {\sc Wen-Bin Zhang}, \emph{Beurling generalized numbers},
Mathematical Surveys and Monographs, 213. American Mathematical Society, Providence, RI, 2016. xi+244 pp.


\bibitem{GGL}{\sc S. M. Gonek, S. W. Graham, Y. Lee}, The Lindelöf Hypothesis for primes is equivalent to the Riemann Hypothesis, \emph{Proc. Amer. Math. Soc.} {\bf 148} (2020), 2863--2875.

\bibitem{Hal-Mittelwerte}{\sc G. Halász}, Über die Mittelwerte multiplikativer zahlentheoretischer Funktionen, \emph{Acta Math. Hungar.} {\bf 19} (1968), 365--404.


\bibitem{Hal-Tur-I}{\sc G. Halász and P. Turán}, On the distribution of roots of Riemann zeta and allied functions, I, \emph{J. Number Theory} {\bf 1} (1969), 121--137.


\bibitem{Hal-Tur-II}{\sc G. Halász and P. Turán}, On the distribution of roots of Riemann zeta and allied functions, II, \emph{Acta Math. Hungar.} {\bf 21} (1970), 403--419.


\bibitem{HB}{\sc D. R. Heath-Brown}, The density of zeros of Dirichlet's L functions, \emph{Canadian J. Math.}, {\bf 31} (1979), 231--240.

\bibitem{H-12}{\sc T. W. Hilberdink},
Generalised prime systems with periodic integer counting function.
\emph{Acta Arith.} {\bf 152} (2012), no. 3, 217--241.


\bibitem{H-5}{\sc T. W. Hilberdink},
Well-behaved Beurling primes and integers.
\emph{J. Number Theory} {\bf 112} (2005), no. 2, 332--344.



\bibitem{Hilberdink-Kazu}{\sc T. Hilberdink, L. Kazyulite}, $\Omega$-result for the remainder term in Beurling's prime number theorem for well-behaved integers. \emph{Acta Arith.} {\bf 208} (2023), no. 1, 69--81.

\bibitem{Ingham} {\sc A. E. Ingham},
{\it The distribution of prime numbers}, Cambridge University
Press, 1932.

\bibitem{Kac1} {\sc J. Kaczorowski}, Results on the distribution of primes, \emph{J. Reine Angew. Math.} {\bf 446} (1994), 89--113.

\bibitem{KP}{\sc J. Kaczorowsky} and {\sc A. Perelli}, Twists, Euler products and a converse theorem for $L$-functions of degree 2, \emph{Ann. Sc. Norm. Super. Pisa Cl. Sci}. (5) {\bf Vol. XIV} (2015), 441--480.

\bibitem{Kadiri} {\sc H. Kadiri},  Une région explicite sans z\'eros pour la fonction $\zeta$ de Riemann. {\em Acta Arith.} {\bf 117} 303--339.

\bibitem{K-97} {\sc J.-P. Kahane}, Sur les nombres premiers généralisés de Beurling. Preuve d'une conjecture de Bateman et Diamond, \emph{J. Théorie Nombres Bordeaux} {\bf 9} (1997) 251–266.

\bibitem{K-98} {\sc J.-P. Kahane}, Le rôle des algebres $A$ de Wiener, $A^\infty$ de Beurling et $H^1$ de Sobolev dans la théorie des nombres premiers généralisés, \emph{Ann. Inst. Fourier} (Grenoble) {\bf 48}(3) (1998) 611--648.

\bibitem{K-99} {\sc J.-P. Kahane},
Un théoreme de Littlewood pour les nombres premiers de Beurling. (French) [A Littlewood theorem for Beurling primes]
\emph{Bull. London Math. Soc.} {\bf 31} (1999), no. 4, 424--430.

\bibitem{Knapowski} {\sc S. Knapowski} On the mean
values of certain functions in prime number theory, {\it Acta
Math. Acad. Sci. Hungar.} {\bf 10} (1959), 375--390.

\bibitem{Knopf} {\sc J. Knopfmacher},  {\em Abstract analytic number theory}, North Holland \& Elsevier, Amsterdam--Oxford \& New York, 1975.  (Second edition: Dover Books on Advanced Mathematics. Dover Publications, Inc., New York, 1990. xii+336 pp.)

\bibitem{L}{\sc E. Landau} \emph{Handbuch der Lehre von der
Verteilung der Primzahlen}, Teubner, Leipzig--Berlin, 1909.

\bibitem{Littlewood} {\sc J. E. Littlewood}, Mathematical notes
(12). An inequality for a sum of cosines, {\it J. London Math.
Soc.} {\bf 12} (1937), 217--222.

\bibitem{Mont} {\sc H. L. Montgomery},  {\em Topics in multiplicative
number theory}, Lecture Notes in Mathematics, {\bf 227}, Springer,
1971.


\bibitem{MT-JNT} {\sc M. J. Mossinghof, T. S. Trudgian}, Nonnegative trigonometric polyanomials and a zero-free region for the Riemann zeta-function, \emph{J. Number Theory} {\bf 157} (2015), 329--349.


\bibitem{H-20}{\sc A. A. Neamah} and {\sc T. W. Hilberdink}
The average order of the Möbius function for Beurling primes.
\emph{Int. J. Number Theory} {\bf 16} (2020), no. 5, 1005--1011.

\bibitem{Pintz1} {\sc J. Pintz}, On the remainder term of the prime number formula. I. On a problem of Littlewood, {\it Acta Arith.} {\bf 36} (1980), 341--365.

\bibitem{Pintz2} {\sc J. Pintz}, On the remainder term of the
prime number formula. II. On a theorem of Ingham, {\it Acta
Arith.} {\bf 37} (1980), 209--220.

\bibitem{Pintz9} {\sc J. Pintz}, Elementary methods in the theory
of $L$-functions IX. Density theorems {\em Acta Arith.} {\bf XLIX}
(1980), 387-394.

\bibitem{PintzProcStekl} {\sc J. Pintz}, Distribution of the zeros of the Riemann zeta function and oscillations of the error term in the asymptotic law of the distribution of prime numbers. (Russian) \emph{Tr. Mat. Inst. Steklova} {\bf 296} (2017), Analiticheskaya i Kombinatornaya Teoriya Chisel, 207--219. English version published in \emph{Proc. Steklov Inst. Math.} {\bf 296} (2017), no. 1, 198--210.


\bibitem{PintzNewDens} {\sc J. Pintz}, Some new density theorems for Dirichlet L-functions. \emph{Number theory week 2017, Banach Center Publ.}, {\bf 118}, 231--244. Polish Acad. Sci. Inst. Math., Warsaw, 2019. See also at {\tt https://arxiv.org/abs/1804.05552}.

\bibitem{Pintz-AMH-Dens} {\sc J. Pintz}, On the density theorem of Halász and Turán. \emph{Acta Math. Hung.}, {\bf 166} (2022), 48--56.

\bibitem{PintzMP} {\sc J. Pintz}, On the mean value of arithmetic error terms, \emph{Math. Pann.}, {\bf 28} (NS 2) (2022) 58--64. See at {\tt https://doi.org/10.1556/314.2022.00007}.

\bibitem{PintzBestDens} {\sc J. Pintz}, Density theorems for Riemann's zeta-function near the line $\Re s=1$. \emph{Acta Arith.} {\bf 208} (2023), no. 1, 1--13.


\bibitem{Rev1} {\sc Sz. Gy. R\'ev\'esz}, Irregularities in the
distribution of prime ideals. I, {\it Studia Sci. Math. Hungar.}
{\bf 18} (1983), 57--67.

\bibitem{Rev2} {\sc Sz. Gy. R\'ev\'esz}, Irregularities in the
distribution of prime ideals. II, {\it Studia Sci. Math. Hungar.}
{\bf 18} (1983), 343--369.

\bibitem{RevPh} {\sc Sz. Gy. R\'ev\'esz}, On a theorem
of Phragm\`en,  {\em Complex analysis and aplications '85,
(Proceedings of the conference held in Varna, Bulgaria, 1985)},
Publ. House Bulgar. Acad. Sci. Sofia, 1986, p. 556--568.

\bibitem{RevAA} {\sc Sz. Gy. R\'ev\'esz}, Effective oscillation
theorems for a general class of real-valued remainder terms, {\em
Acta Arith.} {\bf XLIX} (1988), 482-505.

\bibitem{RevB} {\sc Sz. Gy. R\'ev\'esz}, On Beurling's
prime number theorem,  {\em Periodica Math. Hungar.} {\bf 28}
(1994), no. 3, 195--210.

\bibitem{Rev-L} {\sc Sz. Gy. R\'ev\'esz}, On some extremal problems of Landau. \emph{Serdica Math. J.} {\bf 33} (2007), no. 1, 125--162.

\bibitem{Rev-MP} {\sc Sz. Gy. R\'ev\'esz}, A Riemann-von Mangoldt-type formula for the distribution of Beurling primes. {\it Math. Pann.}, New Series {\bf 27} /NS {\bf 1}/ (2021) 2, 204--232.

\bibitem{Rev-D} {\sc Sz. Gy. R\'ev\'esz}, Density theorems for the Beurling zeta function. \emph{Mathematika}, {\bf 68} (2022), 1045--1072. See at {\tt http://doi.org/10.1112/mtk.12156}.

\bibitem{Rev-One} {\sc Sz. Gy. R\'ev\'esz}, Oscillation of the remainder term in the prime number theorem of Beurling, ``caused by a given $\zeta$-zero". \emph{Int. Math. Res. Not. IMRN} (2023), no. 14, 11752–11790. See also at {\tt https://arxiv.org/abs/arXiv:2202.01837}.

\bibitem{Rev-Many} {\sc Sz. Gy. R\'ev\'esz}, The method of Pintz for the Ingham question about the connection of distribution of $\zeta$-zeros and order of the error in the PNT in the Beurling context. \emph{Michigan Mathematical Journal}, to appear. Advance Publication, 1--46 (2024). DOI: {\tt https://doi.org/10.1307/mmj/20226271}. See also at {\tt https://arxiv.org/abs/2207.00665}.

\bibitem{Schlage-PuchtaVindas} {\sc J.-C. Schlage-Puchta} and {\sc J. Vindas}, The prime number theorem for Beurling's generalized numbers. New cases, \emph{Acta Arith.} {\bf 133}(3) (2012) 293--324.

\bibitem{Stas1} {\sc W. Sta\'s}, \"Uber eine Anwendung
der Methode von Tur\'an auf die Theorie des Restgliedes im
Primidealsatzes, {\it Acta Arith.} {\bf 5} (1959), 179--195.

\bibitem{Stas2} {\sc W. Sta\'s}, \"Uber die Umkehrung eines Satzes
von Ingham, {\it Acta Arith.} {\bf 6} (1961), 435--446.

\bibitem{Stas-Wiertelak-1} {\sc W. Sta\'s} and {\sc K. Wiertelak}, Some estimates in the theory of functions represented Dirichlet's series, {\it Funct. Approx. Comment. Math.} {\bf 1} (1974), 107--111.

\bibitem{Stas-Wiertelak-2} {\sc W. Sta\'s} and {\sc K. Wiertelak}, A comparison of certain remainders connected with prime ideals in ideal classes mod $f$, {\it Funct. Approx. Comment. Math.} {\bf 4} (1976), 99--107.


\bibitem{St} {\sc S. B. Stechkin}, The zeros of the Riemann zeta-function (Russian). \emph{Mat. Zametki} {\bf 8} (1970), 419--429; translation in \emph{Math. Notes} {\bf 8} (1970), 700--700.

%

\bibitem{Turan1} {\sc P. Tur\'an}, On the remainder term
of the prime-number formula, I, {\it Acta Math. Acad. Hungar.}
{\bf 1} (1950), 48--63.

\bibitem{Turan2}
{\sc P. Tur\'an}, On the remainder-term of the prime number
formula. II, {\it Acta Math. Acad. Sci. Hungar.} {\bf 1} (1950),
155--160.

\bibitem{TuranD}{\sc P. Tur\'an}, On the so called density hypothesis
in the theory of zeta-function of Riemann, \emph{Acta Arith.} {\bf
4} (1958), 31-56.

\bibitem{V}{\sc Ch.-J. de la Vall\'ee Poussin}, Sur la
fonction $\zeta(s)$ de Riemann ..., \emph{Memoaire couronn\'es ...
de Belgique}, {\bf 59} No 1 (1899), 1--74.

\bibitem{Vindas12} {\sc J. Vindas}, Chebyshev estimates for Beurling's generalized prime numbers. I, \emph{J. Number Theory} {\bf 132} (2012) 2371--2376.

\bibitem{Vindas13} {\sc J. Vindas}, Chebyshev upper estimates for Beurling's generalized prime numbers.
\emph{Bull. Belg. Math. Soc. Simon Stevin} {\bf 20} (2013), no. 1, 175--180.


\bibitem{Zhang15-IJM} {\sc Wen-Bin Zhang}, Extensions of Beurling's prime number theorem.
\emph{Int. J. Number Theory} {\bf 11} (2015), no. 5, 1589--1616.

\bibitem{Zhang15-MM} {\sc Wen-Bin Zhang}, A proof of a conjecture of Bateman and Diamond on Beurling generalized primes.
\emph{Monatsh. Math.} 176 (2015), no. 4, 637--656.

\bibitem{Zhang7} {\sc Wen-Bin Zhang},
Beurling primes with RH and Beurling primes with large oscillation.
\emph{Math. Ann.} {\bf 337} (2007), no. 3, 671--704.

\bibitem{Zhang93} {\sc Wen-Bin Zhang},
Chebyshev type estimates for Beurling generalized prime numbers. II.
\emph{Trans. Amer. Math. Soc.} {\bf 337} (1993), no. 2, 651--675.

\end{thebibliography}
\end{document}